\documentclass[leqno,a4paper,twoside,12pt]{article}%
\usepackage{amssymb}
\usepackage{amsfonts}
\usepackage{amsmath}
\usepackage{graphicx}%
\providecommand{\U}[1]{\protect\rule{.1in}{.1in}}
\newtheorem{theorem}{Theorem}

\newtheorem{criterion}[theorem]{Criterion}
\newtheorem{definition}[theorem]{Definition}
\newtheorem{example}[theorem]{Example}

\newtheorem{proposition}[theorem]{Proposition}
\newtheorem{remark}[theorem]{Remark}

\newenvironment{proof}[1][Proof]{\noindent\textbf{#1.} }{\ \rule{0.5em}{0.5em}}
\begin{document}

\title{Controllability Metrics on Networks with Linear Decision Process-type
Interactions and Multiplicative Noise}
\author{Tidiane Diallo\footnotemark[1]
\and Dan Goreac\thanks{Universit\'{e} Paris-Est, LAMA (UMR 8050), UPEMLV, UPEC,
CNRS, F-77454, Marne-la-Vall\'{e}e, France, Dan.Goreac@u-pem.fr{}}
\thanks{\textbf{Acknowledgement.} The work of the second author has been
partially supported by he French National Research Agency project PIECE,
number \textbf{ANR-12-JS01-0006.}}}
\maketitle

\begin{abstract}
This paper aims at the study of controllability properties and induced
controllability metrics on complex networks governed by a class of (discrete
time) linear decision processes with multiplicative noise. The dynamics are
given by a couple consisting of a Markov trend and a linear decision process
for which both the "deterministic" and the noise components rely on
trend-dependent matrices. We discuss approximate, approximate null and exact
null-controllability. Several examples are given to illustrate the links
between these concepts and to compare our results with their continuous-time
counterpart (given in \cite{GoreacMartinez2015}). We introduce a class of
backward stochastic Riccati difference schemes (BSRDS) and study their
solvability for particular frameworks. These BSRDS allow one to introduce
Gramian-like controllability metrics. As application of these metrics, we
propose a minimal intervention-targeted reduction in the study of gene networks.

\end{abstract}



\section{Introduction}

We focus on a particular class of discrete-time decision processes described
by a couple denoted $\left(  L,X\right)  $ and consisting of a Markovian trend
and a linearly trend-based updated component. This kind of processes naturally
appear in the study of complex systems (such as regulatory gene networks). In
this setting, the trend component corresponds to a finite family of DNA
configurations which induce regime changes on functional components (usual
proteins) $X.$ Decisions are assumed to be made at expression level in order
to obtain suitable behavior of $X.$ We try to give a mathematical answer to
the following questions. Given a family of possible actions, what are the
minimal interventions to be selected in order to guarantee a targeted
response. Second, how can this be quantified through a metric at the level of
biochemical reactions network ? To answer these questions, we envisage
invariance and Gramian-type descriptions of controllability concepts. This
paper can be seen as a discrete-time counterpart of \cite{GoreacMartinez2015}
in which piecewise deterministic Markov processes of switch type are
considered. Together, the papers cover the two usual points of view over
controlled switch processes with linear updating : the averaged, piecewise
deterministic (macroscopic) perspective (in \cite{GoreacMartinez2015}) and the
marked point process (closer to microscopic perspective in this paper).

The process $L$ is assumed to be a finite-state Markov process on a filtered
probability space $\left(  \Omega,\mathbb{F},\mathbb{P}\right)  ,$ with
transition measure $Q$ and taking its values in $\mathcal{B=}\left\{
e_{1},e_{2},...,e_{p}\right\}  $, for some integer $p\geq2.$ Without loss of
generality, $\Omega$ is set to be the discrete sample space $\Omega
=\mathcal{B}^{\mathbb{N}}$ and we assume the filtration $\mathbb{F=}\left(
\mathcal{F}_{n}\right)  _{n\geq0}$ to be the natural one associated to $L.$
Following \cite{CohenElliot2011ProgrProb}, \cite{Cohen_Elliott_SPA_2010} and
without loss of generality, we assume that $\mathcal{B}$ is the standard basis
of vectors of $%
%TCIMACRO{\U{211d} }%
%BeginExpansion
\mathbb{R}
%EndExpansion
^{p}$. We introduce the martingale $M$ by setting%
\[
M_{n}:=\sum_{k=0}^{n-1}\left[  L_{k+1}-\mathbb{E}\left[  L_{k+1}%
/\mathcal{F}_{k}\right]  \right]  ,
\]
for $1\leq n\leq N,$ with the obvious convention $M_{0}=0$. As usual, we let
$\Delta M_{n}:=M_{n}-M_{n-1},$ for all $n\geq1.$ We will be focusing on a
class of decision processes $X$ on some state space $\mathbb{R}^{m}$, for
$m\geq1$ and controlled by $d$-dimensional processes, for some $d\geq1$. The
evolution is given by linear updating and multiplicative noise
\begin{equation}
\left\{
\begin{array}
[c]{l}%
X_{n+1}^{x,u}\left(  \omega\right)  =A_{n}\left(  \omega\right)  X_{n}%
^{x,u}\left(  \omega\right)  +Bu_{n+1}\left(  \omega\right)  +\sum_{i=1}%
^{p}\left\langle \Delta M_{n+1}\left(  \omega\right)  ,e_{i}\right\rangle
C_{i,n}\left(  \omega\right)  X_{n}^{x,u}\left(  \omega\right)  ,\text{ }\\
X_{0}^{x,u}=x\in%
%TCIMACRO{\U{211d} }%
%BeginExpansion
\mathbb{R}
%EndExpansion
^{m}.
\end{array}
\right.  \label{MDP2}%
\end{equation}
The process $A$ is $\mathbb{R}^{m\times m}$-valued and $\mathbb{F-}$adapted$,$
the matrix $B\in\mathbb{R}^{m\times d},$ and $C_{i}$ are $\mathbb{R}^{m\times
m}$-valued and $\mathbb{F-}$adapted processes for all $1\leq i\leq p$. The $%
%TCIMACRO{\U{211d} }%
%BeginExpansion
\mathbb{R}
%EndExpansion
^{d}$-valued control process $u$ is taken to be square-integrable $\mathbb{F}%
$-predictable. The set of all $\mathbb{F}$-predictable processes is denoted by
$\mathcal{P}red$. Whenever no confusion is at risk, we will drop the
dependency on $\omega$. The reader may want to note that this provides a
slightly more general framework than Markov decision processes since the
coefficients are adapted (i.e. functions of the time parameter $n$ and the
vector $\left(  L_{0},L_{1},...L_{n}\right)  ).$ On the other hand, the
transition measure has a particular form.

The first aim of the paper is to characterize controllability properties for
systems driven by (\ref{MDP2}) i.e. the possibility to direct the process
towards a coherent target. For controlled linear deterministic systems
$\overset{\cdot}{X}_{t}=AX_{t}+Bu_{t}$, the controllability properties are
summarized by the celebrated Kalman criterion stating that $Rank\left[
B\text{ }AB\text{ }A^{2}B\text{ }...A^{m-1}B\right]  =m.$ Similar assertions
are valid for discrete systems $X_{n+1}=AX_{n}+Bu_{n+1}$. \ This can equally
be extended for Markov decision processes driven by non-random coefficients
and additive noise of type $X_{n+1}=AX_{n}+Bu_{n+1}+\xi_{n+1}$. However, for
continuous-time controlled linear systems with multiplicative stochastic
perturbations, this condition is no longer sufficient. For examples pointing
to this direction, the reader is referred to
\cite{Buckdahn_Quincampoix_Tessitore_2006} or \cite{G17} (for Brownian
perturbations), \cite{G10} or \cite{GoreacMartinez2015} (for continuous-time
switch processes).

One can, alternatively, study the dual notion of observability via Hautus'
test as in \cite{Hautus}. The criteria involve algebraic invariance notions
which are independent of the space on which they are studied. For
infinite-dimensional settings, the reader is referred to
\cite{Schmidt_Stern_80}, \cite{Curtain_86}, \cite{russell_Weiss_1994},
\cite{Jacob_Zwart_2001}, \cite{Jacob_Partington_2006}, etc.

In the continuous-time stochastic setting, duality techniques lead to backward
stochastic differential equations (BSDE introduced in \cite{Pardoux_Peng_90}).
With these tools, exact (terminal-) controllability of Brownian-driven control
systems is linked to a full rank condition in \cite{Peng_94}. When the control
is absent from the noise term, one studies approximate controllability, resp.
approximate null-controllability. Invariance criteria are given in
\cite{Buckdahn_Quincampoix_Tessitore_2006} for the control-free noise and
\cite{G17} for the general Brownian setting. In the case when the stochastic
perturbation is of jump-type, exact controllability of continuous-time
processes cannot be achieved. This follows from incompleteness (cf.
\cite{Merton_76}) and one has to concentrate on approximate controllability.

For continuous-time control systems with Brownian noise, approximate and
approximate null-controllability notions coincide (cf. \cite{G17}). This is no
longer the case (see \cite{G1}) when an infinite-dimensional component of
mean-field type governs the Brownian-driven systems. Various methods can be
employed in infinite-dimensional state space Brownian setting leading to
partial results (see \cite{Fernandez_Cara_Garrido_atienza_99},
\cite{Sarbu_Tessitore_2001}, \cite{Barbu_Rascanu_Tessitore_2003}, \cite{G16}).

The main goal of the first part of our paper is to study the controllability
properties of the Markov decision process with linear updating and
multiplicative noise perturbations. It can be seen as a discrete-time
counter-part of \cite{GoreacMartinez2015} and, to some extent,
\cite{Sarbu_Tessitore_2001}. We begin with a duality result between
controllability and observability in Section \ref{Section2.1}. To address
observability, we consider some adjoint process satisfying a backward
difference scheme. Its construction is close to backward stochastic difference
equations (see, for example, \cite{CohenElliot2011ProgrProb},
\cite{Cohen_Elliott_SPA_2010}, \cite{Cohen_Elliott_SIAM_2011}). The first main
result of the paper (Theorem \ref{PropNSC}) gives two characterizations for
approximate null-controllability and a duality criterion for approximate
controllability. It equally states the equivalence between approximate and
exact null-controllability. However, unlike the continuous-time frameworks
(compare with \cite{Buckdahn_Quincampoix_Tessitore_2006} for Brownian systems
and \cite[Section 4.1, Criterion 3]{GoreacMartinez2015} for jump-systems), in
discrete-time, null-controllability does not imply approximate
controllability. This surprising behavior is illustrated in Example
\ref{ExpNullCtrlNotAppCtrl}.

To construct a controllability metric, we concentrate on Gramian-inspired
techniques in Sections \ref{Section2.3} and \ref{Section2.4}. We show in
Example \ref{ExpKalmanNotNullCtrl} that the deterministic Gramian $%
%TCIMACRO{\dsum \limits_{i=1}^{N}}%
%BeginExpansion
{\displaystyle\sum\limits_{i=1}^{N}}
%EndExpansion
A^{i-1}BB^{T}\left(  A^{T}\right)  ^{i-1}$ does not provide a
null-controllability metric. We propose a backward stochastic Riccati
difference scheme (BSRDS) providing the adequate controllability metric. The
link between this BSRDS and null-controllability makes the object of our
second main result (Theorem \ref{TheoremRiccati}). To our best knowledge,
these particular schemes are new to the very rich literature on Riccati
techniques. Let us just mention that Riccati methods in connection to linear
stochastic control problems have been extensively employed in both continuous
(cf. \cite{yong_zhou_99}, etc.) and discrete setting (e.g.
\cite{CostaFragosoMarques}, \cite{CostaOliveira_2012}, etc.). The solvability
of the BSRDS and explicit iterative constructions of the solution in
particular frameworks make the object of Section \ref{Section2.5}. We study
the case of non-random coefficients in Proposition \ref{PropBSRDSNonRandom}.
In Proposition \ref{PropBSRDSC=0}, we state the solvability of BSRDS with
random coefficients in the absence of multiplicative noise. Finally, in
Section \ref{Section2.6}, we show that the invariance techniques developed in
\cite{Buckdahn_Quincampoix_Tessitore_2006} for Brownian perturbations and
adapted to trend-dependent jump-systems in \cite{GoreacMartinez2015} are not
suitable in discrete-time. For non-random coefficients, an invariance
condition (similar to \cite[Criterion 3]{GoreacMartinez2015}) is necessary to
achieve null-controllability (cf. Proposition \ref{PropNecInvariance}).
However, it is not sufficient, as shown in Example \ref{ExpN1NotSuff}.
Concerning the second framework, in absence of multiplicative noise, the
continuous-time condition provided in \cite[Section 4.2, Criterion
4]{GoreacMartinez2015} is neither necessary (see Example \ref{ExpncC=0NotNec})
nor sufficient (Example \ref{ExpncC=0NotSuff}).

The aim of Section \ref{Section3} is to provide a possible application of
controllability metrics to biological networks. The mathematical
considerations are motivated by the notion of (sub)modularity (see
\cite[Section 4]{NemhauserWolseyFisher_78}, \cite{Lovasz:1983aa}, etc.) as
well as the recent applications to power electronic actuator placement in the
preprint \cite{SummersCortesiLygeros_2014}. We describe the optimization
problems appearing when one works with several (possible) control matrices and
wishes to keep controllability features by selecting a minimal dimension of
the control space. To end the section, we give a toy model inspired by
bacteriophage $\lambda$ in \cite{hasty_pradines_dolnik_collins_00} and analyze
different scenarios leading to null-controllability.

Finally, Section \ref{Section4} gathers the proofs of our mathematical assertions.

\section{The Main Concepts and Results}

\subsection{Controllability and Duality\label{Section2.1}}

We begin with recalling the following controllability concepts.

\begin{definition}
i. The system (\ref{MDP2}) is said to be controllable at time $N\geq1$ if for
every initial data $x\in%
%TCIMACRO{\U{211d} }%
%BeginExpansion
\mathbb{R}
%EndExpansion
^{m}$ and every $%
%TCIMACRO{\U{211d} }%
%BeginExpansion
\mathbb{R}
%EndExpansion
^{m}-$valued, $\mathcal{F}_{N}-$measurable square integrable random variable
$\xi$, there exists a predictable control process $u\in\mathcal{P}red$ such
that $X_{N}^{x,u}=\xi,$ $\mathbb{P-}$a.s. (i.e. $X_{N}^{x,u}\left(
\omega\right)  =\xi\left(  \omega\right)  ,$ for $\mathbb{P}$-almost all
$\omega\in\Omega$.)

ii. The system (\ref{MDP2}) is said to be null-controllable at time $N\geq1$
if the previous property holds true for $\xi=0.$

iii. The system (\ref{MDP2}) is said to be approximately controllable at time
$N\geq1$ if for every initial data $x\in%
%TCIMACRO{\U{211d} }%
%BeginExpansion
\mathbb{R}
%EndExpansion
^{m}$ and every $%
%TCIMACRO{\U{211d} }%
%BeginExpansion
\mathbb{R}
%EndExpansion
^{m}-$valued, $\mathcal{F}_{N}-$measurable square integrable random variable
$\xi$ and every $\varepsilon>0,$ there exists a predictable control process
$u^{\varepsilon}\in\mathcal{P}red$ such that $\mathbb{E}\left[  \left\vert
X_{N}^{x,u^{\varepsilon}}-\xi\right\vert ^{2}\right]  \leq\varepsilon$.

iv. The system (\ref{MDP2}) is said to be approximately null-controllable at
time $N\geq1$ if the previous property holds true for $\xi=0.$
\end{definition}

To the decision process (\ref{MDP2}), one can associate an adjoint process (or
an adjoint couple) as follows. For every $\mathcal{F}_{N}$-measurable square
integrable random variable $\xi,$ we introduce the adjoint couple $\left(
Y^{N,\xi},Z^{N,\xi}\right)  $ $\ $consisting in an $%
%TCIMACRO{\U{211d} }%
%BeginExpansion
\mathbb{R}
%EndExpansion
^{m}$-valued (resp. $%
%TCIMACRO{\U{211d} }%
%BeginExpansion
\mathbb{R}
%EndExpansion
^{m\times p}$-valued) adapted process by setting%
\begin{equation}
\left\{
\begin{array}
[c]{l}%
Y_{N}^{N,\xi}:=\xi,\text{ }\\
Y_{n}^{N,\xi}:=A_{n}^{T}\mathbb{E}\left[  Y_{n+1}^{N,\xi}/\mathcal{F}%
_{n}\right]  +\sum_{i=1}^{p}C_{i,n}^{T}Z_{n}^{N,\xi}\mathbb{E}\left[
\left\langle \Delta M_{n+1},e_{i}\right\rangle \Delta M_{n+1}/\mathcal{F}%
_{n}\right]  ,\text{ }\\
\text{where }Y_{n+1}^{N,\xi}=\mathbb{E}\left[  Y_{n+1}^{N,\xi}/\mathcal{F}%
_{n}\right]  +Z_{n}^{N,\xi}\Delta M_{n+1},\text{ for all }0\leq n\leq N-1.
\end{array}
\right.  \label{AdjointCouple}%
\end{equation}
The existence (and uniqueness up to an equivalence in the sense of
\cite[Definition 2]{Cohen_Elliott_SPA_2010}) of processes $Z$ satisfying the
last property is standard. We refer the interested reader to \cite[Corollary
1]{Cohen_Elliott_SPA_2010} or \cite[Corollary 3.1.1]{CohenElliot2011ProgrProb}
and references therein.

The first result of our paper provides the following characterization of controllability.

\begin{theorem}
\label{PropNSC}i) The system (\ref{MDP2}) is approximately null-controllable
in time $N>0$ if and only if every solution $\left(  Y_{n}^{N,\xi}%
,Z_{n}^{N,\xi}\right)  $ of the scheme (\ref{AdjointCouple}) satisfying
$\mathbb{E}\left[  Y_{n}^{N,\xi}/\mathcal{F}_{n-1}\right]  \in\ker\left(
B^{T}\right)  ,$ $\mathbb{P}$-a.s., for all $1\leq n\leq N,$ equally satisfies
$Y_{0}^{N,\xi}=0,$ $\mathbb{P}$-a.s.

ii) The system (\ref{MDP2}) is approximately controllable in time $N>0$ if and
only if every solution $\left(  Y_{n}^{N,\xi},Z_{n}^{N,\xi}\right)  $ of the
scheme (\ref{AdjointCouple}) satisfying $\mathbb{E}\left[  Y_{n}^{N,\xi
}/\mathcal{F}_{n-1}\right]  \in\ker\left(  B^{T}\right)  ,$ $\mathbb{P}$-a.s.,
for all $1\leq n\leq N,$ equally satisfies $\mathbb{E}\left[  \xi
/\mathcal{F}_{n-1}\right]  =0,$ $\mathbb{P}$-a.s.

iii) The system (\ref{MDP2}) is approximately null-controllable in time $N>0$
if and only if it is (exactly) null-controllable (in time $N>0$). The
necessary and sufficient condition for null-controllability is the existence
of some constant $k>0$ such that%
\begin{equation}
\left\vert Y_{0}^{N,\xi}\right\vert ^{2}\leq k\mathbb{E}\left[  \sum_{n=1}%
^{N}\left\langle BB^{T}\mathbb{E}\left[  Y_{n}^{N,\xi}/\mathcal{F}%
_{n-1}\right]  ,\mathbb{E}\left[  Y_{n}^{N,\xi}/\mathcal{F}_{n-1}\right]
\right\rangle \right]  , \label{0ctrlInequality}%
\end{equation}
for all $\left(  Y_{n}^{N,\xi},Z_{n}^{N,\xi}\right)  $ satisfying
(\ref{AdjointCouple}).
\end{theorem}

The proof is postponed to Section \ref{Section4}. The first two assertions are
proven by taking convenient controllability operators and identifying their
duals. The third assertion makes use of these duals and the finite-dimensional setting.

\subsection{An Alternative Characterization and an Example\label{Section2.2}}

When the linear coefficient $A$ is invertible, we are able to restate the
null-controllability criterion given in Theorem \ref{PropNSC} iii) by
interpreting the adjoint couple as a decision process (where the second
component of the couple is an arbitrary predictable control). We also give an
example showing that, in the context of discrete processes,
null-controllability is, in all generality, strictly weaker that approximate-controllability.

From now on, unless stated otherwise, the matrix $A_{n}\left(  \omega\right)
$ is assumed to be invertible for $\mathbb{P-}$almost all $\omega\in\Omega$
and all $n\geq0$. The reader will note that $\left(  Y^{N,\xi},Z^{N,\xi
}\right)  $ given by (\ref{AdjointCouple}) can be interpreted in connection to
a (forward) decision process by picking $v_{n+1}:=Z_{n}^{N,\xi}$ and setting%
\begin{equation}
\left\{
\begin{array}
[c]{l}%
y_{0}:=Y_{0}^{N,\xi},\text{ }y_{0}^{y_{0},v}=y_{0},\\
y_{n+1}^{y_{0},v}:=\left[  A_{n}^{T}\right]  ^{-1}\left(  y_{n}^{y_{0},v}%
-\sum_{i=1}^{p}C_{i,n}^{T}v_{n+1}\mathbb{E}\left[  \left\langle \Delta
M_{n+1},e_{i}\right\rangle \Delta M_{n+1}/\mathcal{F}_{n}\right]  \right)
+v_{n+1}\Delta M_{n+1},
\end{array}
\right.  \label{MDPdual}%
\end{equation}
for all $0\leq n\leq N-1.$

\begin{remark}
1. When $A$ is not invertible, the admissible controls should be such that
\[
y_{n}^{y_{0},v}-\sum_{i=1}^{p}C_{i,n}^{T}v_{n+1}\mathbb{E}\left[  \left\langle
\Delta M_{n+1},e_{i}\right\rangle \Delta M_{n+1}/\mathcal{F}_{n}\right]
\in\operatorname{Im}\left(  A_{n}^{T}\right)  ,
\]
where $\operatorname{Im}$ stands for the image of the linear operator.
Nevertheless, the connection is still preserved.

2. The adjoint process is motivated by the duality techniques in the Brownian
case (e.g. in \cite{Buckdahn_Quincampoix_Tessitore_2006}). These arguments
concern backward stochastic differential equations. The specialization of this
concept to discrete-time processes is the notion of backward stochastic
difference equation (e.g. \cite{Cohen_Elliott_SPA_2010},
\cite{Cohen_Elliott_SIAM_2011}). In view of the essential bijection
requirement (cf. \cite[Theorem 2]{Cohen_Elliott_SPA_2010}, \cite[Theorem
1.2]{Cohen_Elliott_SIAM_2011}), asking for $A$ to be invertible does not
appear to be a drawback.
\end{remark}

In this framework, the third assertion of Theorem \ref{PropNSC} can be
interpreted as follows.

\begin{criterion}
\label{CritNullCtrl}The system (\ref{MDP2}) is approximately (and exactly)
null-controllable if and only if there exists some $k>0$ such that for every
$y_{0}\in%
%TCIMACRO{\U{211d} }%
%BeginExpansion
\mathbb{R}
%EndExpansion
^{m}$ and every $\mathbb{F-}$predictable, $%
%TCIMACRO{\U{211d} }%
%BeginExpansion
\mathbb{R}
%EndExpansion
^{m\times p}$-valued sequence $\left(  v_{n}\right)  _{1\leq n\leq N},$ one
has
\[
\left\vert y_{0}\right\vert ^{2}\leq k\mathbb{E}\left[  \sum_{n=0}%
^{N-1}\left\vert B^{T}\mathbb{E}\left[  y_{n+1}^{y_{0},v}/\mathcal{F}%
_{n}\right]  \right\vert ^{2}\right]  ,
\]
where $y^{y_{0},v}$ is the decision process defined by (\ref{MDPdual}).
\end{criterion}

In the continuous-time framework, when the controlled linear systems are
driven by non-random and homogeneous coefficients (i.e. systems for which $A$
and $C$ are constant matrices independent of $n$), it has been proven that
approximate null-controllability and approximate controllability are
equivalent. The reader is referred to \cite[Theorem 1.3]%
{Buckdahn_Quincampoix_Tessitore_2006} (for Brownian setting) and to
\cite[Theorem 2.2]{G10} and \cite[Criterion 3]{GoreacMartinez2015} for jump
systems. The following example shows that, in the case of discrete-time
processes, one can have (exact) null-controllability without having
approximate controllability.

\begin{example}
\label{ExpNullCtrlNotAppCtrl}To this purpose, let us take $p=2$ and the
transition matrix $Q=\left(
\begin{array}
[c]{cc}%
\frac{1}{2} & \frac{1}{2}\\
\frac{1}{2} & \frac{1}{2}%
\end{array}
\right)  .$ We consider the time horizon $N=2,$ the state space dimension
$m=2$ and the control space dimension $d=1.$ Moreover, we consider
\[
A_{n}=\left(
\begin{array}
[c]{cc}%
0 & 1\\
1 & 0
\end{array}
\right)  ,\text{ }B=\left(
\begin{array}
[c]{c}%
0\\
1
\end{array}
\right)  ,\text{ }C_{i,n}=\left(
\begin{array}
[c]{cc}%
0 & \left(  -1\right)  ^{i+1}\\
0 & 0
\end{array}
\right)  ,\text{ for }i\in\left\{  1,2\right\}  \text{ and }n\geq0.
\]
Then, the decision process (\ref{MDP2}) becomes%
\[
X_{0}^{x,u}=x=\left(
\begin{array}
[c]{c}%
x_{1}\\
x_{2}%
\end{array}
\right)  ,\text{ }X_{1}^{x,u}=\left(
\begin{array}
[c]{c}%
\left(  1+\left\langle L_{1},e_{1}-e_{2}\right\rangle \right)  x_{2}\\
x_{1}+u_{1}%
\end{array}
\right)  ,\text{ }X_{2}^{x,u}=\left(
\begin{array}
[c]{c}%
\left(  x_{1}+u_{1}\right)  \left(  1+\left\langle L_{2},e_{1}-e_{2}%
\right\rangle \right) \\
\left(  1+\left\langle L_{1},e_{1}-e_{2}\right\rangle \right)  x_{2}+u_{2}%
\end{array}
\right)  .
\]
We consider $u_{1}=-x_{1}$ and $u_{2}=-\left(  1+\left\langle L_{1}%
,e_{1}-e_{2}\right\rangle \right)  x_{2}$ to conclude that the system is
exactly null-controllable in 2 steps. Nevertheless, by considering
$\xi:=\left(
\begin{array}
[c]{c}%
\left\langle L_{2},e_{1}-e_{2}\right\rangle \\
0
\end{array}
\right)  ,$ one has
\[
\mathbb{E}\left[  \left\vert X_{2}^{x,u}-\xi\right\vert ^{2}\right]
\geq\mathbb{E}\left[  \left[  \left(  x_{1}+u_{1}\right)  +\left(  x_{1}%
+u_{1}-1\right)  \left\langle L_{2},e_{1}-e_{2}\right\rangle \right]
^{2}\right]  \geq\frac{1}{2},
\]
for all $x\in%
%TCIMACRO{\U{211d} }%
%BeginExpansion
\mathbb{R}
%EndExpansion
^{2}$ and all $u_{1}\in%
%TCIMACRO{\U{211d} }%
%BeginExpansion
\mathbb{R}
%EndExpansion
$. The system turns out not to be approximately controllable at time $N=2$.
(This holds true for $N\geq2$ by simply replacing, in the definition of $\xi$,
$L_{2}$ with $L_{N}$. In fact, we have proven something stronger: the system
is not even exactly terminal controllable; see \cite{Peng_94} for a comparison
with the Brownian setting).
\end{example}

\begin{remark}
To prove null-controllability, one can, alternatively, rely on Criterion
\ref{CritNullCtrl}. For $y_{0}=\left(
\begin{array}
[c]{c}%
y_{0}^{1}\\
y_{0}^{2}%
\end{array}
\right)  \in%
%TCIMACRO{\U{211d} }%
%BeginExpansion
\mathbb{R}
%EndExpansion
^{2}$ and a family of $\mathbb{F}$-predictable, $%
%TCIMACRO{\U{211d} }%
%BeginExpansion
\mathbb{R}
%EndExpansion
^{2\times2}$-valued controls $v_{1}=\left(
\begin{array}
[c]{cc}%
v_{1}^{1,1} & v_{1}^{1,2}\\
v_{1}^{2,1} & v_{1}^{2,2}%
\end{array}
\right)  ,$ $v_{2}=\left(
\begin{array}
[c]{cc}%
v_{2}^{1,1} & v_{2}^{1,2}\\
v_{2}^{2,1} & v_{2}^{2,2}%
\end{array}
\right)  $,\ simple (yet fastidious) computations show that%
\[
\mathbb{E}\left[  \sum_{n=0}^{1}\left\vert B^{T}\mathbb{E}\left[
y_{n+1}^{y_{0},v}/\mathcal{F}_{n}\right]  \right\vert ^{2}\right]  =\left(
y_{0}^{1}\right)  ^{2}+\mathbb{E}\left[  \left(  y_{0}^{2}+\left(
\left\langle L_{1},e_{1}-e_{2}\right\rangle -\frac{1}{2}\right)  \frac
{v_{1}^{1,1}-v_{1}^{1,2}}{2}\right)  ^{2}\right]  \geq\frac{1}{2}\left\vert
y_{0}\right\vert ^{2}.
\]
One concludes to the exact null-controllability by applying Criterion
\ref{CritNullCtrl}.
\end{remark}

\subsection{The Deterministic Controllability Metric Is
Insufficient\label{Section2.3}}

\bigskip A simple glance at the inequality in Criterion \ref{CritNullCtrl}
allows one to infer that the right-hand term (i.e. $\mathbb{E}\left[
\sum_{n=0}^{N-1}\left\vert B^{T}\mathbb{E}\left[  y_{n+1}^{y_{0}%
,v}/\mathcal{F}_{n}\right]  \right\vert ^{2}\right]  $) should provide a
quadratic function of the initial data $y_{0}$ which is positive-definite when
the initial system is null-controllable. In the deterministic framework ($C=0$
and non-random, constant $A$), the controllability criterion is given by the
celebrated Kalman condition
\[
Rank\left[  B\text{ }AB\text{ }A^{2}B\text{ }...A^{N-1}B\right]  =m.
\]
By taking the expectation in (\ref{MDP2}), one gets that this condition is
necessary for the system governed by non-random $A,B,C$ to be approximately
null-controllable. This condition is equivalent (cf. \cite[Theorem
2.16]{CostaFragosoMarques} (assertions 2 and 3) to the matrix
\begin{equation}
p_{0}:=%
%TCIMACRO{\dsum \limits_{i=1}^{N}}%
%BeginExpansion
{\displaystyle\sum\limits_{i=1}^{N}}
%EndExpansion
A^{i-1}BB^{T}\left(  A^{T}\right)  ^{i-1}\text{ } \label{metricDet}%
\end{equation}
being of full rank. In this case, the controllability (pseudo)norm given by $%
%TCIMACRO{\U{211d} }%
%BeginExpansion
\mathbb{R}
%EndExpansion
^{m}\ni y_{0}\mapsto\left(  \left\langle p_{0}y_{0},y_{0}\right\rangle
\right)  ^{\frac{1}{2}}$ is a norm. So the obvious question one asks oneself
is whether the same norm is actually sufficient to get stochastic
null-controllability. The answer is negative (see example hereafter).. Unlike
the additive case, the presence of multiplicative noise induces a change in
the controllability condition. This is not surprising for our reader familiar
with the stochastic framework. Indeed, the invariance conditions
characterizing approximate null-controllability in \cite[Theorem
1.3]{Buckdahn_Quincampoix_Tessitore_2006} or \cite[Criterion 3]%
{GoreacMartinez2015} involve the stochastic component (i.e. $C$). The
following example shows that, in the discrete framework, one may have Kalman's
condition and not achieve the null-controllability of the stochastic system.

\begin{example}
\label{ExpKalmanNotNullCtrl}To this purpose, let us take $p=2$ and the
transition matrix $Q=\left(
\begin{array}
[c]{cc}%
\frac{1}{2} & \frac{1}{2}\\
\frac{1}{2} & \frac{1}{2}%
\end{array}
\right)  .$ We consider the time horizon $N=2,$ the state space dimension
$m=2$ and the control space dimension $d=1.$ Moreover, we consider
\[
A_{n}=\left(
\begin{array}
[c]{cc}%
0 & 1\\
1 & 0
\end{array}
\right)  ,\text{ }C_{i,n}=\left(  -1\right)  ^{i+1}A_{n}\text{, }B=\left(
\begin{array}
[c]{c}%
0\\
1
\end{array}
\right)  ,\text{ for }i\in\left\{  1,2\right\}  \text{ and }n\geq1.
\]
We also drop the dependency on $n$ in $A$ and $C$. Then, $Rank\left[  B\text{
}AB\right]  =Rank\left[
\begin{array}
[c]{cc}%
0 & 1\\
1 & 0
\end{array}
\right]  =2.$ However, by taking the initial condition $x=\left(
\begin{array}
[c]{c}%
1\\
0
\end{array}
\right)  ,$ one gets
\[
X_{2}^{x,u}=\left(
\begin{array}
[c]{c}%
\left(  1+\left\langle L_{2},e_{1}-e_{2}\right\rangle \right)  \left(
1+\left\langle L_{1},e_{1}-e_{2}\right\rangle +u_{1}\right) \\
u_{2}%
\end{array}
\right)  .
\]
As consequence, for any (predictable) choice of the control $u$,
\[
\mathbb{E}\left[  \left\vert X_{2}^{x,u}\right\vert ^{2}\right]  =\left(
u_{1}^{2}+\left(  2+u_{1}\right)  ^{2}\right)  +\mathbb{E}\left[  u_{2}%
^{2}\right]  \geq2.
\]
Therefore, independently of the predictable control we use, we are not able to
drive the process $X$ from $x$ to $0$ even though Kalman's condition is satisfied.
\end{example}

\subsection{A Stochastic Controllability Metric\label{Section2.4}}

In view of Criterion \ref{CritNullCtrl}, we associate, to every point $y$ in $%
%TCIMACRO{\U{211d} }%
%BeginExpansion
\mathbb{R}
%EndExpansion
^{m}$ the controllability (pseudo)norm
\begin{equation}
\left\Vert y\right\Vert _{ctrl}^{2}:=\inf_{\left(  v\right)  _{1\leq n\leq
N}\text{ }\mathbb{F}\text{-predictable}}\mathbb{E}\left[  \sum_{n=0}%
^{N-1}\left\vert B^{T}\mathbb{E}\left[  y_{n+1}^{y_{0},v}/\mathcal{F}%
_{n}\right]  \right\vert ^{2}\right]  . \label{ctrl_norm}%
\end{equation}
The previous example shows that the associated metric is a stochastic one and,
in general, it does not reduce to the deterministic Gramian. Nevertheless, one
would very much like to have something which is close to the $p_{0}$ matrix in
(\ref{metricDet}). In this section, we thrive to provide a (pseudo)metric
which is more explicit than (\ref{ctrl_norm}).

To this purpose, let us analyze the following matrix scheme. We set, for
$\varepsilon>0,$ $P_{N}^{\varepsilon}=0\in%
%TCIMACRO{\U{211d} }%
%BeginExpansion
\mathbb{R}
%EndExpansion
^{m\times m}$ and proceed by writing, for all $0\leq n\leq N-1,$%
\[
P_{n+1}^{\varepsilon}=\mathbb{E}\left[  P_{n+1}^{\varepsilon}/\mathcal{F}%
_{n}\right]  +Q_{n}^{\varepsilon}\text{ }diag\left(  \Delta M_{n+1}\right)  ,
\]
where, by convention,
\begin{equation}
diag\left(  \Delta M_{n+1}\right)  =\left(
\begin{array}
[c]{cccc}%
\Delta M_{n+1} & 0 & ... & 0\\
0 & \Delta M_{n+1} & ... & 0\\
... & ... & ... & \\
0 & 0 & ... & \Delta M_{n+1}%
\end{array}
\right)  \in%
%TCIMACRO{\U{211d} }%
%BeginExpansion
\mathbb{R}
%EndExpansion
^{mp\times m}. \label{diagM}%
\end{equation}
The existence and uniqueness of such $Q_{n}^{\varepsilon}\in%
%TCIMACRO{\U{211d} }%
%BeginExpansion
\mathbb{R}
%EndExpansion
^{m\times mp}$ is obtained by applying the martingale representation theorem
(see, for example \cite[Corollary 1]{Cohen_Elliott_SPA_2010} or
\cite[Corollary 3.1.1]{CohenElliot2011ProgrProb} and references therein) to
the columns of $P_{n+1}^{\varepsilon}.$ We proceed by setting
\begin{equation}
P_{n}^{\varepsilon}=A_{n}^{-1}\left(  \mathbb{E}\left[  P_{n+1}^{\varepsilon
}/\mathcal{F}_{n}\right]  +BB^{T}\right)  \left[  A_{n}^{T}\right]
^{-1}-\alpha_{n,\varepsilon}^{T}\eta_{n,\varepsilon}^{-1}\alpha_{n,\varepsilon
}, \label{RiccatiGeneralIt}%
\end{equation}
where $\alpha_{n,\varepsilon}=\left(
\begin{array}
[c]{c}%
\alpha_{n,\varepsilon}^{1}\\
...\\
\alpha_{n,\varepsilon}^{p}%
\end{array}
\right)  \in%
%TCIMACRO{\U{211d} }%
%BeginExpansion
\mathbb{R}
%EndExpansion
^{mp\times m},$ $\eta_{n,\varepsilon}=\left(
\begin{array}
[c]{cccc}%
\eta_{n,\varepsilon}^{1,1} & \eta_{n,\varepsilon}^{1,2} & ... & \eta
_{n,\varepsilon}^{1,p}\\
\eta_{n,\varepsilon}^{2,1} & \eta_{n,\varepsilon}^{2,2} & ... & \eta
_{n,\varepsilon}^{2,p}\\
... & ... & ... & ...\\
\eta_{n,\varepsilon}^{p,1} & \eta_{n,\varepsilon}^{p,2} & ... & \eta
_{n,\varepsilon}^{p,p}%
\end{array}
\right)  \in%
%TCIMACRO{\U{211d} }%
%BeginExpansion
\mathbb{R}
%EndExpansion
^{mp\times mp}$ are given by
\begin{align*}
&  \alpha_{n,\varepsilon}^{j}:=-Q_{n}^{\varepsilon}\mathbb{E}\left[
\left\langle \Delta M_{n+1},e_{j}\right\rangle diag\left(  \Delta
M_{n+1}\right)  /\mathcal{F}_{n}\right]  \left[  A_{n}^{T}\right]  ^{-1}\\
&  \text{ \ \ \ \ \ \ \ }+\sum_{1\leq i\leq p}\mathbb{E}\left[  \left\langle
\Delta M_{n+1},e_{j}\right\rangle \left\langle \Delta M_{n+1},e_{i}%
\right\rangle /\mathcal{F}_{n}\right]  C_{i,n}A_{n}^{-1}\left(  \mathbb{E}%
\left[  P_{n+1}^{\varepsilon}/\mathcal{F}_{n}\right]  +BB^{T}\right)  \left[
A_{n}^{T}\right]  ^{-1}\text{ and}%
\end{align*}%
\begin{align*}
&  \eta_{n,\varepsilon}^{j,k}:=\\
&  \varepsilon\delta_{j,k}I_{m\times m}+\frac{1}{2}Q_{n}^{\varepsilon
}\mathbb{E}\left[  \left\langle \Delta M_{n+1},e_{k}\right\rangle \left\langle
\Delta M_{n+1},e_{j}\right\rangle diag\left(  \Delta M_{n+1}\right)
/\mathcal{F}_{n}\right] \\
&  +\frac{1}{2}\mathbb{E}\left[  \left\langle \Delta M_{n+1},e_{k}%
\right\rangle \left\langle \Delta M_{n+1},e_{j}\right\rangle \left(
diag\left(  \Delta M_{n+1}\right)  \right)  ^{T}/\mathcal{F}_{n}\right]
\left(  Q_{n}^{\varepsilon}\right)  ^{T}\\
&  -\frac{1}{2}\sum_{1\leq i\leq p}Q_{n}^{\varepsilon}\mathbb{E}\left[
\left\langle \Delta M_{n+1},e_{i}\right\rangle \left\langle \Delta
M_{n+1},e_{j}\right\rangle /\mathcal{F}_{n}\right]  \mathbb{E}\left[
\left\langle \Delta M_{n+1},e_{k}\right\rangle diag\left(  \Delta
M_{n+1}\right)  /\mathcal{F}_{n}\right]  \left[  A_{n}^{T}\right]
^{-1}C_{i,n}^{T}\\
&  -\frac{1}{2}\sum_{1\leq i\leq p}C_{i,n}A_{n}^{-1}\mathbb{E}\left[
\left\langle \Delta M_{n+1},e_{i}\right\rangle \left\langle \Delta
M_{n+1},e_{k}\right\rangle /\mathcal{F}_{n}\right]  \mathbb{E}\left[
\left\langle \Delta M_{n+1},e_{j}\right\rangle \left(  diag\left(  \Delta
M_{n+1}\right)  \right)  ^{T}/\mathcal{F}_{n}\right]  \left(  Q_{n}%
^{\varepsilon}\right)  ^{T}\\
&  +\mathbb{E}\left[  \left\langle \Delta M_{n+1},e_{k}\right\rangle
\left\langle \Delta M_{n+1},e_{j}\right\rangle /\mathcal{F}_{n}\right]
\mathbb{E}\left[  P_{n+1}^{\varepsilon}/\mathcal{F}_{n}\right] \\
&  +\sum_{1\leq i,i^{\prime}\leq p}\left(  \left.
\begin{array}
[c]{c}%
\mathbb{E}\left[  \left\langle \Delta M_{n+1},e_{j}\right\rangle \left\langle
\Delta M_{n+1},e_{i}\right\rangle /\mathcal{F}_{n}\right]  \times
\mathbb{E}\left[  \left\langle \Delta M_{n+1},e_{k}\right\rangle \left\langle
\Delta M_{n+1},e_{i^{\prime}}\right\rangle /\mathcal{F}_{n}\right]  \times\\
\times C_{i^{\prime},n}A_{n}^{-1}\left(  \mathbb{E}\left[  P_{n+1}%
^{\varepsilon}/\mathcal{F}_{n}\right]  +BB^{T}\right)  \left[  A_{n}%
^{T}\right]  ^{-1}C_{i,n}^{T}%
\end{array}
\right.  \right)  ,
\end{align*}
for all $1\leq j,k\leq p.$ Here, $\delta_{j,k}$ stands for the classical
Kronecker delta ($1,$ if $j=k$ and $0,$ otherwise). While it is clear that
$\eta_{n,\varepsilon}$ is symmetric , it is (a lot) less obvious to ask for
$\eta_{n,\varepsilon}$ to be positive. We will show in some particular cases
that this stochastic Riccati-type difference equation is solvable and provide
explicit construction for $P$ and $Q$. For the time being, let us assume that,
for every $\varepsilon>0,$ such a solution exists. The second main result of
our paper is the following characterization of the null-controllability.

\begin{theorem}
\label{TheoremRiccati}i. The system (\ref{MDP2}) is (approximately)
null-controllable if and only if
\[
\underset{\varepsilon\rightarrow0+}{\lim\inf}P_{0}^{\varepsilon}\text{ is a
positive-definite, symmetric matrix.}%
\]

ii. The controllability (pseudo)norm given by (\ref{ctrl_norm}) satisfies
\[
\left\Vert y_{0}\right\Vert _{ctrl}^{2}=\underset{\varepsilon\rightarrow
0}{\lim\inf}\left\langle P_{0}^{\varepsilon}y_{0},y_{0}\right\rangle .
\]

\end{theorem}

The proof is postponed to Section \ref{Section4}. The construction of
$P^{\varepsilon}$ is tailor-made to infer a recurrence on the terms
$\left\langle P_{n}^{\varepsilon}y_{n}^{y_{0},v},y_{n}^{y_{0},v}\right\rangle
$. To conclude, one applies Criterion \ref{CritNullCtrl}.

\begin{remark}
\label{remOptimality}i. This result is the discrete-time version of
\cite[Theorem 3.4]{Sarbu_Tessitore_2001}.

ii. A simple look at the proof (see (\ref{optimalityPeps})) shows that
\[
\left\langle P_{0}^{\varepsilon}y_{0},y_{0}\right\rangle =\inf_{\left(
v_{n}\right)  _{1\leq n\leq N}\text{ }\mathbb{F}\text{-predictable}}\left(
\varepsilon\sum_{n=0}^{N-1}\mathbb{E}\left[  \left\vert v_{n+1}\right\vert
^{2}\right]  +\mathbb{E}\left[  \sum_{n=0}^{N-1}\left\vert B^{T}%
\mathbb{E}\left[  y_{n+1}^{y_{0},v}/\mathcal{F}_{n}\right]  \right\vert
^{2}\right]  \right)
\]
and the optimal control realizing this minimum is given in feedback form by
\[
v_{n+1}^{opt}=\eta_{n,\varepsilon}^{-1}\alpha_{n,\varepsilon}y_{n}%
^{y_{0},v^{opt}}.
\]
Nevertheless, due to the structure of $\alpha,$ $y_{n}^{y_{0},v^{opt}}$ might
not be a Markov process$.$
\end{remark}

\subsection{Solvability of the Backward Stochastic Riccati Difference Scheme
(BSRDS)\label{Section2.5}}

The aim of this subsection is to provide two simple cases in which the
backward stochastic Riccati scheme admits explicit solutions. One of the
simplest frameworks for our trend component is the one in which the martingale
is generated by independent and identically distributed variables. In other
words, we assume $L_{n+1}$ to be independent of $\mathcal{F}_{n}$ for all
$n\geq0$ and $L_{n}$ has the same law as $L_{0}.$ Then $\left(  \left\langle
L_{n},e_{i}\right\rangle \right)  _{n\geq1}$ are independent Bernoulli
distributed with some parameter $q_{i}>0$ (independent of $n$) and such that
$\sum_{i=1}^{p}q_{i}=1.$

The first setting is when $A$ and $C$ consist of sequences of non-random
matrices. In other words, the randomness may only come from the martingale
induced by the trend component $L$. In this case, we get the following result
of solvability of the BSRDS.

\begin{proposition}
[non-random coefficients case]\label{PropBSRDSNonRandom}We assume that $L_{n}$
are independent, identically distributed random variables on $\left\{
e_{1},e_{2},...,e_{p}\right\}  $ and denote by
\[
q_{i}=\mathbb{P}\left(  L_{1}=e_{i}\right)  >0,\text{ for every }1\leq i\leq
p.
\]
Moreover, we assume that
\[
A_{n}\text{, }C_{n}\in%
%TCIMACRO{\U{211d} }%
%BeginExpansion
\mathbb{R}
%EndExpansion
^{m\times m},\text{ for all }n\geq0
\]
are sequences of (non-random) matrices. Then, for every $\varepsilon>0,$ the
BSRDS (\ref{RiccatiGeneralIt}) admits a (unique) solution given by a
positive-semidefinite sequence $\left(  P_{n}^{\varepsilon}\right)  _{0\leq
n\leq N}$ and $Q^{\varepsilon}=0.$ This solution is given by
\begin{equation}
\left\{
\begin{array}
[c]{l}%
P_{N}^{\varepsilon}=0,\\
\alpha_{n,\varepsilon}=\mathcal{C}_{n}A_{n}^{-1}\left(  P_{n+1}^{\varepsilon
}+BB^{T}\right)  \left[  A_{n}^{T}\right]  ^{-1},\text{ }\\
\eta_{n,\varepsilon}=\varepsilon I_{mp\times mp}+\left(  q_{j}\left(
\delta_{j,k}-q_{k}\right)  P_{n+1}^{\varepsilon}\right)  _{1\leq j,k\leq
p}+\mathcal{C}_{n}A_{n}^{-1}\left(  P_{n+1}^{\varepsilon}\mathcal{+}%
BB^{T}\right)  \left[  \mathcal{C}_{n}A_{n}^{-1}\right]  ^{T},\\
P_{n}^{\varepsilon}=A_{n}^{-1}\left(  P_{n+1}^{\varepsilon}+BB^{T}\right)
\left[  A_{n}^{T}\right]  ^{-1}-\alpha_{n,\varepsilon}^{T}\eta_{n,\varepsilon
}^{-1}\alpha_{n,\varepsilon},
\end{array}
\right.  \label{RiccatiNonRandomCoeff}%
\end{equation}
where $\mathcal{C}_{n}:\mathcal{=}\left[
\begin{array}
[c]{c}%
\sum_{l=1}^{p}\left(  \delta_{1,l}-q_{1}\right)  q_{l}C_{l,n}\\
\sum_{l=1}^{p}\left(  \delta_{2,l}-q_{2}\right)  q_{l}C_{l,n}\\
...\\
\sum_{l=1}^{p}\left(  \delta_{p,l}-q_{p}\right)  q_{l}C_{l,n}%
\end{array}
\right]  .$
\end{proposition}

The proof follows by (descending) induction and is postponed to Section
\ref{Section4}.

\begin{remark}
\label{RemOpenLoop}i. We emphasize that we deal here with a difference
equation and not with an algebraic Riccati equation and this is why one does
not need further conditions on the Popov matrix.

ii. The Riccati difference equation given by (\ref{RiccatiNonRandomCoeff}) is
obviously deterministic. Then, a glance at the optimal control in Remark
\ref{remOptimality} shows that $v_{n+1}^{opt}=\eta_{n,\varepsilon}^{-1}%
\delta_{n,\varepsilon}y_{n}^{y_{0},v^{opt}}$ is not only predictable but a
deterministic function of time $n$ and the state of the process $y_{n}$.
Therefore, the infimum in $\left\langle P_{0}^{\varepsilon}y_{0}%
,y_{0}\right\rangle $ can be taken over open-loop control strategies. Of
course, similar assertions hold true for the controllability (pseudo)norm. As
a by-product the process $y$ constructed with open-loop controls is Markovian.
\end{remark}

The second case in which solving the backward stochastic Riccati difference
equation is reduced to deterministic arguments is when $C=0.$ Under this
assumption, the BSRDS becomes%

\begin{equation}
\left\{
\begin{array}
[c]{l}%
P_{n}^{\varepsilon}:=A_{n}^{-1}\left(  \mathbb{E}\left[  P_{n+1}^{\varepsilon
}/\mathcal{F}_{n}\right]  +BB^{T}\right)  \left[  A_{n}^{T}\right]
^{-1}-\alpha_{n,\varepsilon}^{T}\eta_{n,\varepsilon}^{-1}\alpha_{n,\varepsilon
},\\
\alpha_{n,\varepsilon}^{j}:=-q_{j}Q_{n}^{\varepsilon}diag\left(
%TCIMACRO{\tsum \limits_{k=1}^{p}}%
%BeginExpansion
{\textstyle\sum\limits_{k=1}^{p}}
%EndExpansion
\left(  \delta_{j,k}-q_{k}\right)  e_{k}\right)  \left[  A_{n}^{T}\right]
^{-1},\\
\eta_{n,\varepsilon}^{j,k}:=\varepsilon\delta_{j,k}I_{m\times m}+\left(
\delta_{j,k}-q_{k}\right)  q_{j}\mathbb{E}\left[  P_{n+1}^{\varepsilon
}/\mathcal{F}_{n}\right]  +\frac{1}{2}Q_{n}^{\varepsilon}diag\left(
%TCIMACRO{\tsum \limits_{l=1}^{p}}%
%BeginExpansion
{\textstyle\sum\limits_{l=1}^{p}}
%EndExpansion
m_{j,k,l}e_{l}\right) \\
+\frac{1}{2}\left(  diag\left(
%TCIMACRO{\tsum \limits_{l=1}^{n}}%
%BeginExpansion
{\textstyle\sum\limits_{l=1}^{n}}
%EndExpansion
m_{j,k,l}e_{l}\right)  \right)  ^{T}\left(  Q_{n}^{\varepsilon}\right)  ^{T}.
\end{array}
\right.  \label{RiccatiC=0}%
\end{equation}
for all $1\leq j,k\leq p.$ Here,%
\[
m_{j,k,l}=q_{l}\left(  q_{j}-\delta_{j,l}\right)  \left(  q_{k}-\delta
_{k,l}\right)  -q_{l}\left(  \delta_{j,k}-q_{j}\right)  q_{k}.
\]
Let us recall that the $diag$ matrices are of type $%
%TCIMACRO{\U{211d} }%
%BeginExpansion
\mathbb{R}
%EndExpansion
^{mp\times m}$ and are constructed as in (\ref{diagM}).

When $A_{n}=A\left(  n,L_{n}\right)  ,$ for all $0\leq n\leq N$, where $A$ is
some $%
%TCIMACRO{\U{211d} }%
%BeginExpansion
\mathbb{R}
%EndExpansion
^{m\times m}$-valued deterministic function of time and trend, we set
\begin{equation}
p_{N}^{\varepsilon}=0_{m\times m},\text{ }q_{N}^{\varepsilon}=0_{m\times m}
\label{pqN}%
\end{equation}
and construct, for, $n<N-1,$
\begin{equation}
\left\{
\begin{array}
[c]{l}%
p_{n+1}^{\varepsilon}:=%
%TCIMACRO{\tsum \limits_{l=1}^{p}}%
%BeginExpansion
{\textstyle\sum\limits_{l=1}^{p}}
%EndExpansion
q_{l}A^{-1}\left(  n+1,e_{l}\right)  \left(  p_{n+2}^{\varepsilon}%
+BB^{T}-q_{n+2}^{\varepsilon}\right)  \left(  A^{-1}\left(  n+1,e_{l}\right)
\right)  ^{T}\text{,}\\
q_{n+1}^{\varepsilon}=\overline{\alpha}_{n,\varepsilon}^{T}\eta_{n,\varepsilon
}^{-1}\overline{\alpha}_{n,\varepsilon},
\end{array}
\right.  \label{pqn+1}%
\end{equation}
where
\begin{equation}
\left\{
\begin{array}
[c]{l}%
\overline{\alpha}_{n,\varepsilon}^{j}=\left[
%TCIMACRO{\tsum \limits_{l=1}^{p}}%
%BeginExpansion
{\textstyle\sum\limits_{l=1}^{p}}
%EndExpansion
q_{l}\left(  \delta_{j,l}-q_{j}\right)  A^{-1}\left(  n+1,e_{l}\right)
\left(  p_{n+2}^{\varepsilon}+BB^{T}-q_{n+2}^{\varepsilon}\right)  \left(
A^{-1}\left(  n+1,e_{l}\right)  \right)  ^{T}\right]  ,\\
\alpha_{n,\varepsilon}^{j}=-\overline{\alpha}_{n,\varepsilon}^{j}\left[
A_{n}^{T}\right]  ^{-1}.\\
\eta_{n,\varepsilon}^{j,k}=\varepsilon\delta_{j,k}I_{m\times m}\\
\text{ \ \ \ \ \ }+%
%TCIMACRO{\tsum \limits_{l=1}^{p}}%
%BeginExpansion
{\textstyle\sum\limits_{l=1}^{p}}
%EndExpansion
q_{l}\left(  q_{j}-\delta_{j,l}\right)  \left(  q_{k}-\delta_{k,l}\right)
A^{-1}\left(  n+1,e_{l}\right)  \left(  p_{n+2}^{\varepsilon}+BB^{T}%
-q_{n+2}^{\varepsilon}\right)  \left(  A^{-1}\left(  n+1,e_{l}\right)
\right)  ^{T},
\end{array}
\right.  \label{alphaetaC=0}%
\end{equation}
for all $1\leq j,k\leq p$.

The main result in this framework gives the solvability of the BSRDS
(\ref{RiccatiC=0}).

\begin{proposition}
[the case without multiplicative noise]\label{PropBSRDSC=0}We assume that
$L_{n}$ are independent, identically distributed random variables on $\left\{
e_{1},e_{2},...,e_{p}\right\}  $ and denote by
\[
q_{i}=\mathbb{P}\left(  L_{1}=e_{i}\right)  >0,\text{ for every }1\leq i\leq
p.
\]
Moreover, we assume that $A_{n}=A\left(  n,L_{n}\right)  $ where $A$ is some $%
%TCIMACRO{\U{211d} }%
%BeginExpansion
\mathbb{R}
%EndExpansion
^{m\times m}$-valued deterministic function of time and trend. Then, for every
$\varepsilon>0,$ the following assertions hold true :

i. The matrices $\left(  p_{n}^{\varepsilon}\right)  _{0\leq n\leq N}$ and
$\left(  q_{n}^{\varepsilon}\right)  _{0\leq n\leq N}$ $\ $given by
(\ref{pqN}, \ref{pqn+1}) are positive-semidefinite and
\begin{equation}
p_{n}^{\varepsilon}\geq q_{n}^{\varepsilon},\text{ for all }0\leq n\leq N
\label{pngeqn}%
\end{equation}

ii. The solution of (\ref{RiccatiC=0}) is given by the couple $\left(
P^{\varepsilon},Q\right)  \in%
%TCIMACRO{\U{211d} }%
%BeginExpansion
\mathbb{R}
%EndExpansion
^{m\times m}\times%
%TCIMACRO{\U{211d} }%
%BeginExpansion
\mathbb{R}
%EndExpansion
^{m\times mp}$ defined by $P_{N}^{\varepsilon}=0$ and $Q_{N-1}=0$
\begin{equation}
P_{n}^{\varepsilon}=A^{-1}\left(  n,L_{n}\right)  \left(  p_{n+1}%
^{\varepsilon}+BB^{T}-q_{n+1}^{\varepsilon}\right)  \left(  A^{-1}\left(
n,L_{n}\right)  \right)  ^{T},\text{ for all }0\leq n\leq N-1
\label{Pn=f(pn,qn)}%
\end{equation}
and%
\begin{align}
&  Q_{n-1}=\left[  Q_{n-1,1}^{\cdot,\cdot},Q_{n-1,2}^{\cdot,\cdot
},...,Q_{n-1,m}^{\cdot,\cdot}\right]  \in%
%TCIMACRO{\U{211d} }%
%BeginExpansion
\mathbb{R}
%EndExpansion
^{m\times p}\times%
%TCIMACRO{\U{211d} }%
%BeginExpansion
\mathbb{R}
%EndExpansion
^{m\times p}\times...\times%
%TCIMACRO{\U{211d} }%
%BeginExpansion
\mathbb{R}
%EndExpansion
^{m\times p},\text{ where}\label{Qn=f(pn,qn)}\\
&  Q_{n-1,j}^{i,l}=\left[  A^{-1}\left(  n,e_{l}\right)  \left(
p_{n+1}^{\varepsilon}+BB^{T}-q_{n+1}^{\varepsilon}\right)  \left(
A^{-1}\left(  n,e_{l}\right)  \right)  ^{T}\right]  _{i,j},\nonumber
\end{align}
for all $1\leq i,j\leq m,$ $1\leq l\leq p,$ $1\leq n\leq N-1$.
\end{proposition}

\begin{remark}
i. The dependence of $\varepsilon$ in $Q$ has been dropped to simplify notation.

ii. When one further assumes that $A$ is non-random, the elements
$Q_{n-1,k}^{\cdot,l}$ are independent of $l.$ Hence,
\[
Q_{n-1,k}^{\cdot,\cdot}\Delta M_{n+1}=Q_{n-1,k}^{\cdot,1}\sum_{l=1}^{p}\left(
\left\langle L_{n+1},e_{l}\right\rangle -q_{l}\right)  =0_{m\times
1}=0_{m\times p}\Delta M_{n+1},
\]
i.e. $Q_{n-1}$ is equivalent (see, for example \cite[Definition 2]%
{Cohen_Elliott_SPA_2010}) to $0_{m\times mp}.$ This is consistent with the
results in our non-random framework.
\end{remark}

To end this subsection, let us take a look at the case when $C\left(
\cdot\right)  =0$ and $A_{n}$ is a non-random matrix. Using the second point
of the previous remark, one gets $\alpha=0_{m\times mp}$ and $\eta
_{n,\varepsilon}=\varepsilon I_{mp\times mp}$. Hence, one only has to solve
the following (deterministic, $\varepsilon$-independent scheme) :%

\[
p_{n}=A_{n}^{-1}\left(  p_{n+1}+BB^{T}\right)  \left(  A_{n}^{T}\right)
^{-1},\text{ }p_{N}=0_{m\times m}.
\]
In this framework, we get the following criterion.

\begin{criterion}
\label{CritDeterministic}Let us assume that $C\left(  \cdot\right)  =0$ and
$A_{n}$ is a non-random matrix, for all $n\geq1.$ Then, the system
(\ref{MDP2}) is null-controllable in time $N>0$ if and only if the matrix
\[
p_{0}^{N}=%
%TCIMACRO{\tsum \limits_{n=0}^{N-1}}%
%BeginExpansion
{\textstyle\sum\limits_{n=0}^{N-1}}
%EndExpansion
\left[  \left(
%TCIMACRO{\tprod \limits_{k=0}^{n}}%
%BeginExpansion
{\textstyle\prod\limits_{k=0}^{n}}
%EndExpansion
A_{k}^{-1}\right)  BB^{T}\left(
%TCIMACRO{\tprod \limits_{k=0}^{n}}%
%BeginExpansion
{\textstyle\prod\limits_{k=0}^{n}}
%EndExpansion
A_{k}^{-1}\right)  ^{T}\right]
\]
has full rank.
\end{criterion}

\begin{remark}
(a) If $A$ does not depend on $n,$ $p_{0}$ is of full rank if and only if
$A^{N}p_{0}\left(  A^{T}\right)  ^{N}$ is of full rank and one gets the
classical condition $p_{0}=%
%TCIMACRO{\tsum \limits_{n=0}^{N-1}}%
%BeginExpansion
{\textstyle\sum\limits_{n=0}^{N-1}}
%EndExpansion
A^{n}BB^{T}\left(  A^{T}\right)  ^{n}$ is of full rank. This classical
condition (and its equivalence with the usual Kalman rank condition) can be
found, for instance in \cite[Theorem 2.16]{CostaFragosoMarques} (assertions 2
and 3).

(b) In the case without multiplicative noise, one can establish (in the same
way as in the reference cited in \cite[Theorem 2.16]{CostaFragosoMarques}),
the equivalence between the Criterion \ref{CritDeterministic} and the rank
condition given for communicating classes in \cite[Theorem 2]{JiChizeck1988}.
Reasoning on the cannonical space, this criterion can be extended for more
general systems with no multiplicative noise (see also Remark
\ref{RemDimension}).
\end{remark}

\subsection{When Continuous-time Intuition Fails to Work\label{Section2.6}}

As we have seen in Example \ref{ExpNullCtrlNotAppCtrl}, the
null-controllability metric is given by a strictly weaker condition than that
of controllability. The reader acquainted with the continuous-time
characterizations of approximate controllability
(\cite{Buckdahn_Quincampoix_Tessitore_2006} or \cite{G17} for Brownian
perturbations, \cite{G10} or \cite{GoreacMartinez2015} for continuous-time
jump processes) may wonder whether alternative approaches based on invariance
concepts can be adapted to this discrete framework. The aim of this section is
to compare our approach with the algebraic conditions given in
\cite{GoreacMartinez2015} for continuous-time processes presenting a similar
construction. We will consider both the non-random coefficients setting and
the behavior of the system lacking multiplicative noise.

In the case of non-random coefficients, the method developed in
\cite{Buckdahn_Quincampoix_Tessitore_2006} for Brownian perturbations and
adapted to trend-dependent jump-systems in \cite{GoreacMartinez2015} consists
in obtaining invariance criteria starting from (\ref{MDPdual}). We will prove
that the analogous condition is still necessary in order to have
null-controllability. Nevertheless, this condition is strictly weaker than the
one exhibited in Theorem \ref{TheoremRiccati} (see Example \ref{ExpN1NotSuff}%
). Concerning the second framework, in absence of multiplicative noise, the
continuous-time condition provided in \cite[Section 4.2, Criterion
4]{GoreacMartinez2015} is neither necessary (see Example \ref{ExpncC=0NotNec})
nor sufficient (Example \ref{ExpncC=0NotSuff}).

Throughout the subsection, we assume that $L_{n}$ are independent, identically
distributed random variables on $\left\{  e_{1},e_{2},...,e_{p}\right\}  $ and
denote by $q_{i}=\mathbb{P}\left(  L_{1}=e_{i}\right)  >0,$ for every $1\leq
i\leq p.$

\subsubsection{The Non-Random Coefficients Case\label{Section2.6.1}}

To simplify the arguments, we concentrate on the time-homogeneous framework
(i.e. $A_{n}=A\in%
%TCIMACRO{\U{211d} }%
%BeginExpansion
\mathbb{R}
%EndExpansion
^{m\times m},$ $C_{i,n}=C_{i}\in%
%TCIMACRO{\U{211d} }%
%BeginExpansion
\mathbb{R}
%EndExpansion
^{m\times m}$, for all $n\geq0).$ In this setting, the dual decision process
(\ref{MDPdual}) becomes%
\[
y_{n+1}^{y_{0},v}:=\left[  A^{T}\right]  ^{-1}\left(  y_{n}^{y_{0}%
,v}-\overset{p}{\underset{l=1}{\sum}}\mathcal{C}^{T}\left(  l\right)
v_{n+1}e_{l}\right)  +\overset{p}{\underset{l=1}{\sum}}\left\langle \Delta
M_{n+1},e_{l}\right\rangle v_{n+1}e_{l},\text{ }y_{0}^{y_{0},v}=y_{0},
\]
where
\begin{equation}
\mathcal{C}\left(  j\right)  :\mathcal{=}\sum_{k=1}^{p}\left(  \delta
_{j,k}-q_{j}\right)  q_{k}C_{k}, \label{CcalHomogeneousNonRandom}%
\end{equation}
for every $1\leq j\leq p$.

In \cite{GoreacMartinez2015}, the study of controllability properties is
conducted using some invariance properties with respect to the dual decision
process. We recall the following invariance notions.

\begin{definition}
We consider a linear operator $\mathcal{A\in}%
%TCIMACRO{\U{211d} }%
%BeginExpansion
\mathbb{R}
%EndExpansion
^{m\times m}$ and a family $\left(  \mathcal{C}_{i}\right)  _{1\leq i\leq
t}\subset%
%TCIMACRO{\U{211d} }%
%BeginExpansion
\mathbb{R}
%EndExpansion
^{m\times m}$.

(a) A set $V\subset%
%TCIMACRO{\U{211d} }%
%BeginExpansion
\mathbb{R}
%EndExpansion
^{m}$ is said to be $\mathcal{A}$- invariant if $\mathcal{A}V\subset V.$

(b) A set $V\subset%
%TCIMACRO{\U{211d} }%
%BeginExpansion
\mathbb{R}
%EndExpansion
^{m}$ is said to be $\left(  \mathcal{A};\mathcal{C}\right)  $- invariant if
$\mathcal{A}V\subset V+\operatorname{Im}\mathcal{C}_{1}+\operatorname{Im}%
\mathcal{C}_{2}+...+\operatorname{Im}\mathcal{C}_{t},$ where
$\operatorname{Im}$ stands for the image of the linear operators.

(c) A set $V\subset%
%TCIMACRO{\U{211d} }%
%BeginExpansion
\mathbb{R}
%EndExpansion
^{m}$ is said to be $\left(  \mathcal{A};\mathcal{C}\right)  $- strictly
invariant if
\[
\mathcal{A}V\subset V+\operatorname{Im}\left(  \mathcal{C}_{1}\Pi_{V}\right)
+\operatorname{Im}\left(  \mathcal{C}_{2}\Pi_{V}\right)
+...+\operatorname{Im}\left(  \mathcal{C}_{p}\Pi_{V}\right)  ,
\]
where $\Pi_{V}$ denotes the orthogonal projection onto $V$.

(d) A set $V\subset%
%TCIMACRO{\U{211d} }%
%BeginExpansion
\mathbb{R}
%EndExpansion
^{n}$ is said to be feedback $\left(  \mathcal{A};\mathcal{C}\right)  $-
invariant if there exists a family of linear operators $\left(  \mathcal{F}%
_{i}\right)  _{1\leq i\leq t}\subset%
%TCIMACRO{\U{211d} }%
%BeginExpansion
\mathbb{R}
%EndExpansion
^{m\times m}$ such that $\left(  \mathcal{A}+\sum_{i=1}^{t}\mathcal{C}%
_{i}\mathcal{F}_{i}\right)  V\subset V$ (i.e. $V$ is $\mathcal{A}+\sum
_{i=1}^{t}\mathcal{C}_{i}\mathcal{F}_{i}$- invariant).
\end{definition}

The following condition is necessary in order to have null-controllability.

\begin{proposition}
\label{PropNecInvariance}If the system (\ref{MDP2}) is null controllable at
time $N$, then, by setting
\begin{equation}
\left.
\begin{array}
[c]{c}%
\mathcal{V}^{N,N}=\ker\left(  B^{T}\right)  \text{ and computing, for }0\leq
k<N,\\
\mathcal{V}^{k,N}\text{ to be the largest subspace of the kernel }\ker\left(
B^{T}\right)  \text{ which is }\\
\left(  \left[  A^{T}\right]  ^{-1};\left(  \mathcal{C}\left(  1\right)
A^{-1}\right)  ^{T}\Pi_{\mathcal{V}^{k+1,N}},...,\left(  \mathcal{C}\left(
p\right)  A^{-1}\right)  ^{T}\Pi_{\mathcal{V}^{k+1,N}}\right)
\text{-invariant},
\end{array}
\right.  \tag{N1}\label{N1}%
\end{equation}
the space $\mathcal{V}^{0,N}$ is reduced to $\left\{  0\right\}  $.
\end{proposition}

The reader should compare this with \cite[Section 4.1, Criterion
3]{GoreacMartinez2015}. The proof of this result is very similar to that of
\cite[Section 3.1.2, Proposition 2]{GoreacMartinez2015} and is postponed to
Section \ref{Section4}.

\begin{remark}
\label{RemDimension}i. Both the assertion and the proof can be extended to
non-homogeneous systems providing the complete analogue of \cite[Section
3.1.2, Proposition 2]{GoreacMartinez2015}.

ii. Much as in the continuous-time framework (see \cite[Section 4.1, Criterion
3]{GoreacMartinez2015}), one can prove the equivalence between the following

(a) condition (\ref{N1}) holds true;

(b) every solution of (\ref{MDPdual}) satisfying $B^{T}y_{n}^{y_{0},v}=0$,
$\mathbb{P-}a.s.$ for all $n\geq0$ is such that $y_{0}=0.$

iii. It is obvious that the family $\left(  \mathcal{V}^{k,N}\right)  _{0\leq
k\leq N}$ is increasing in $k$ (and $N)$ and, thus, this condition only needs
to be checked for $N=m.$ Indeed, for $N\geq m$, the space $\mathcal{V}^{0,N}$
is nothing else than the largest subspace of $\ker\left(  B^{T}\right)  $
which is $\left(  \left[  A^{T}\right]  ^{-1};\left(  \mathcal{C}\left(
1\right)  A^{-1}\right)  ^{T},...,\left(  \mathcal{C}\left(  p\right)
A^{-1}\right)  ^{T}\right)  $-strictly invariant$.$ In all generality, for
systems in which $A_{n}=A\left(  L_{n}\right)  $ and $C_{i,n}=C_{i}\left(
L_{n}\right)  ,$ one defines similarly $\mathcal{C}\left(  i,L_{n}\right)  $
and families of subspaces $\mathcal{V}_{l}^{k,N}$ which are
\[
\left(  \left[  A^{T}\left(  l\right)  \right]  ^{-1};\left(  \mathcal{C}%
\left(  1,l\right)  A^{-1}\left(  l\right)  \right)  ^{T}\Pi_{\mathcal{V}%
_{1}^{k+1,N}},...,\left(  \mathcal{C}\left(  p,l\right)  A^{-1}\left(
l\right)  \right)  ^{T}\Pi_{\mathcal{V}_{p}^{k+1,N}}\right)  -invariant.
\]
The monotonicity still holds but, since $\mathcal{V}_{l}^{k,N}$ is given with
respect to all $\left(  \mathcal{V}_{i}^{k+1,N}\right)  _{1\leq i\leq p},$ in
order to be sure that the spaces no longer change, one has to wait for
$N=m^{p}.$
\end{remark}

Nevertheless, unlike the continuous-time framework, the condition (\ref{N1})
is weaker than the null-controllability property. Let us, once again, look at
the following example.

\begin{example}
\label{ExpN1NotSuff}We consider $p=2$ and the transition matrix $Q=\left(
\begin{array}
[c]{cc}%
\frac{1}{2} & \frac{1}{2}\\
\frac{1}{2} & \frac{1}{2}%
\end{array}
\right)  ,$ the horizon $N=2,$ the state space dimension $m=2$ and the control
space dimension $d=1.$ Moreover, we consider
\[
A=\left(
\begin{array}
[c]{cc}%
0 & 1\\
1 & 0
\end{array}
\right)  ,\text{ }C_{i}=\left(  -1\right)  ^{i+1}A\text{, }B=\left(
\begin{array}
[c]{c}%
0\\
1
\end{array}
\right)  ,\text{ for }i\in\left\{  1,2\right\}  .
\]
Then, according to (\ref{CcalHomogeneousNonRandom}), $\mathcal{C}\left(
i\right)  =\frac{\left(  -1\right)  ^{i+1}}{2}A$, for $i\in\left\{
1,2\right\}  .$ Moreover, $\ker\left(  B^{T}\right)  =\left\{  \left(
\begin{array}
[c]{c}%
x\\
0
\end{array}
\right)  ,\text{ }x\in%
%TCIMACRO{\U{211d} }%
%BeginExpansion
\mathbb{R}
%EndExpansion
\right\}  .$ If $x\in%
%TCIMACRO{\U{211d} }%
%BeginExpansion
\mathbb{R}
%EndExpansion
$ is such that, for some $x^{\prime},x^{\prime\prime}\in%
%TCIMACRO{\U{211d} }%
%BeginExpansion
\mathbb{R}
%EndExpansion
,$%
\[
\left[  A^{T}\right]  ^{-1}\left(
\begin{array}
[c]{c}%
x\\
0
\end{array}
\right)  +\left(  \mathcal{C}\left(  1\right)  A^{-1}\right)  ^{T}\left(
\begin{array}
[c]{c}%
x^{\prime}\\
0
\end{array}
\right)  +\left(  \mathcal{C}\left(  2\right)  A^{-1}\right)  ^{T}\left(
\begin{array}
[c]{c}%
x^{\prime\prime}\\
0
\end{array}
\right)  =\left(
\begin{array}
[c]{c}%
\frac{x^{\prime}-x^{\prime\prime}}{2}\\
x
\end{array}
\right)  \in\ker\left(  B^{T}\right)  ,
\]
then it follows that $x=0.$ This means that condition (\ref{N1}) is satisfied.
However, as shown in Example \ref{ExpKalmanNotNullCtrl}, the decision system
driven by $A,B$ and $C$ is not null-controllable. Thus, in all generality, for
discrete-time processes, the condition (\ref{N1}) does not guarantee null-controllability.
\end{example}

\subsubsection{The Case Without Multiplicative Noise (C=0)\label{Section2.6.2}%
}

In the continuous-time framework, the necessary and sufficient condition for
approximate null-controllability of continuous switch systems (equally when
$C=0,$ see \cite[Section 4.2, Criterion 4]{GoreacMartinez2015}) reads%
\begin{equation}
\left(  A_{n},B\right)  \text{ is controllable for all }n.
\label{ncC=0continuous}%
\end{equation}
Unlike the continuous-time setting, we will see that this condition is neither
necessary nor sufficient.

We begin with an example showing that (\ref{ncC=0continuous}) may hold without
implying the null-controllability of the discrete system.

\begin{example}
\label{ExpncC=0NotNec}We consider the state space dimension $m=3$ and the
control space dimension $d=1.$ Moreover, we consider
\[
A_{2n+1}=\left(
\begin{array}
[c]{ccc}%
0 & 0 & 1\\
1 & 0 & 0\\
0 & 1 & 0
\end{array}
\right)  ,\text{ }A_{2n}=\left(
\begin{array}
[c]{ccc}%
0 & 1 & 0\\
0 & 0 & 1\\
1 & 0 & 0
\end{array}
\right)  ,\text{ }B=\left(
\begin{array}
[c]{c}%
1\\
0\\
0
\end{array}
\right)  .
\]
It is obvious that Kalman's condition is satisfied for both $A_{2n}$ and
$A_{2n+1}.$ However, by computing $p_{0}^{N}$ (see Criterion
\ref{CritDeterministic})$,$ one gets
\end{example}

\[
p_{0}^{N}=%
%TCIMACRO{\tsum \limits_{n=0}^{N-1}}%
%BeginExpansion
{\textstyle\sum\limits_{n=0}^{N-1}}
%EndExpansion
\left[  \left(
%TCIMACRO{\tprod \limits_{k=0}^{n}}%
%BeginExpansion
{\textstyle\prod\limits_{k=0}^{n}}
%EndExpansion
A_{k}^{-1}\right)  BB^{T}\left(
%TCIMACRO{\tprod \limits_{k=0}^{n}}%
%BeginExpansion
{\textstyle\prod\limits_{k=0}^{n}}
%EndExpansion
A_{k}^{-1}\right)  ^{T}\right]  =\left(
\begin{array}
[c]{ccc}%
\left\lfloor \frac{N}{2}\right\rfloor  & 0 & 0\\
0 & \left\lfloor \frac{N+1}{2}\right\rfloor  & 0\\
0 & 0 & 0
\end{array}
\right)
\]
which is obviously not invertible for any $N\geq1$. Here, $\left\lfloor
\cdot\right\rfloor $ denotes the largest integer that does not exceed the
argument (floor function).

But null-controllability may hold without having (\ref{ncC=0continuous}) for
any $A_{n}.$

\begin{example}
\label{ExpncC=0NotSuff}We consider the state space dimension $m=3$ and the
control space dimension $d=1.$ Moreover, we consider
\[
A_{2n+1}=\left(
\begin{array}
[c]{ccc}%
0 & 1 & 0\\
1 & 0 & 0\\
0 & 0 & 1
\end{array}
\right)  ,\text{ }A_{2n}=\left(
\begin{array}
[c]{ccc}%
1 & 0 & 0\\
0 & 0 & 1\\
0 & 1 & 0
\end{array}
\right)  ,\text{ for }n\geq0\text{, }B=\left(
\begin{array}
[c]{c}%
1\\
0\\
0
\end{array}
\right)  .
\]
Then $Rank\left[  B\text{ }A_{2n+1}B\text{ }A_{2n+1}^{2}B\right]  =2$ and
$Rank\left[  B\text{ }A_{2n}B\text{ }A_{2n}^{2}B\right]  =1.$ Nevertheless,
for $N=4,$%
\[
p_{0}^{4}=%
%TCIMACRO{\tsum \limits_{n=0}^{3}}%
%BeginExpansion
{\textstyle\sum\limits_{n=0}^{3}}
%EndExpansion
\left[  \left(
%TCIMACRO{\tprod \limits_{k=0}^{n}}%
%BeginExpansion
{\textstyle\prod\limits_{k=0}^{n}}
%EndExpansion
A_{k}^{-1}\right)  BB^{T}\left(
%TCIMACRO{\tprod \limits_{k=0}^{n}}%
%BeginExpansion
{\textstyle\prod\limits_{k=0}^{n}}
%EndExpansion
A_{k}^{-1}\right)  ^{T}\right]  =\left(
\begin{array}
[c]{ccc}%
1 & 0 & 0\\
0 & 1 & 0\\
0 & 0 & 2
\end{array}
\right)
\]
is of full rank such that, using Criterion \ref{CritDeterministic}, the system
is null-controllable at time $N=4.$ The reader may equally want to note that
the controllability condition does not hold true at $N^{\prime}=3=m,$ the
dimension of our state space.
\end{example}

\begin{remark}
i. These counterexamples are of particular relevance for the way one considers
the models in systems biology (see next section). It shows that targeted
design depends on the discrete or continuous modelisation.

ii. In this case (with no multiplicative noise), one can prove, by reasoning
on the cannonical space (as in \cite[Theorem 3]{JiChizeck1988}, or, again, as
in \cite[Section 4.2, Criterion 4]{GoreacMartinez2015}) that
null-controllability is characterized by Criterion \ref{CritDeterministic}
holding true on every feasible sample path. We emphasize that (as it is the
case in the proof of \cite[Section 4.2, Criterion 4]{GoreacMartinez2015}), the
explicit feedback control is given with respect to this controllability
Grammian (cf. \cite[Theorem 3 ii.]{JiChizeck1988}).
\end{remark}

\section{A Minimal Intervention-Targeted Application in Biological
Networks\label{Section3}}

The aim of this section is to provide a possible application of the previous
mathematical tools to biological networks. The mathematical considerations are
motivated by the notion of (sub)modularity (see \cite[Section 4]%
{NemhauserWolseyFisher_78}, \cite{Lovasz:1983aa}, etc.) as well as the recent
applications to power electronic actuator placement in the preprint
\cite{SummersCortesiLygeros_2014}. We describe the optimization problems
appearing when one works with several (possible) control matrices. To end the
section, we give a toy model inspired by bacteriophage $\lambda$ in
\cite{hasty_pradines_dolnik_collins_00}.

\subsection{Intervention Scenarios}

Up until now, the control matrix $B$ has been fixed. We are going to envisage
some scenarios translated by several control matrices. Each fundamental matrix
allows one-dimensional controls. By putting together some of these matrices
(say $d$), we get an $m\times d$ control matrix taking into account
$d$-dimensional controls. Of course, in the case in which several
configurations allow null controllability, it would be interesting if we were
able to find a minimal $d$ (lowest dimension for control processes) giving the
null controllability.

We begin with noting that the (pseudo)norm (\ref{ctrl_norm}) will explicitly
depend on the control matrix $B$ and will be denoted by $\left\Vert
\cdot\right\Vert _{ctrl,B}$. Similar, whenever $P_{0}^{\varepsilon}$ given by
(\ref{RiccatiGeneralIt}) exists, it is written as $P_{0}^{\varepsilon}\left(
B\right)  $. We define
\[
\left\Vert B\right\Vert _{ctrl}^{spec}:=\inf_{y\neq0}\frac{\left\Vert
y\right\Vert _{ctrl,B}}{\left\Vert y\right\Vert }\text{ and }\left\Vert
B\right\Vert _{ctrl}^{rank}:=Rank\left(  \underset{\varepsilon\rightarrow
0}{\lim\inf}P_{0}^{\varepsilon}\left(  B\right)  \right)
\]
It is obvious that the system (\ref{MDP2}) is controllable using $B$ iff
$\left\Vert B\right\Vert _{ctrl}^{spec}>0$ or, equivalently, iff $\left\Vert
B\right\Vert _{ctrl}^{rank}=m.$

A basic intervention scenario is a column vector $b\subset%
%TCIMACRO{\U{211d} }%
%BeginExpansion
\mathbb{R}
%EndExpansion
^{m}$ allowing one-dimensional controls and specifying the weight of this
control in the state component. In other words, we consider the system
controlled by $B=b$ and with $d=1$ in (\ref{MDP2}). Given a family of $r\in%
%TCIMACRO{\U{2115} }%
%BeginExpansion
\mathbb{N}
%EndExpansion
^{\ast}$ intervention scenarios $\left\{  b_{1},b_{2},...,b_{r}\right\}
\subset%
%TCIMACRO{\U{211d} }%
%BeginExpansion
\mathbb{R}
%EndExpansion
^{m}$, for every $\mathcal{I=}\left\{  i_{1},i_{2},...i_{\left\vert
\mathcal{I}\right\vert }\right\}  \mathcal{\subset}\left\{  1,...r\right\}  $
one constructs $B\left(  \mathcal{I}\right)  =\left[  b_{i_{1}}%
,...,b_{i_{\left\vert \mathcal{I}\right\vert }}\right]  $. We introduce the
following two concepts.

\begin{definition}
\label{DefMEI}1) A set $\mathcal{I}$ is called minimal spectral-efficient
intervention if the following assertions hold simultaneously:

(i) one has $\left\Vert B\left(  \mathcal{I}\right)  \right\Vert
_{ctrl}^{spec}>0$;

(ii) for every $\mathcal{J\subset}\left\{  1,...r\right\}  $ such that
$\left\vert \mathcal{J}\right\vert <\left\vert \mathcal{I}\right\vert ,$ one
has $\left\Vert B\left(  \mathcal{J}\right)  \right\Vert _{ctrl}^{spec}=0;$

(iii) for every $\mathcal{J\subset}\left\{  1,...r\right\}  $ such that
$\left\vert \mathcal{J}\right\vert =\left\vert \mathcal{I}\right\vert ,$ one
has $\left\Vert B\left(  \mathcal{J}\right)  \right\Vert _{ctrl}^{spec}%
\leq\left\Vert B\left(  \mathcal{I}\right)  \right\Vert _{ctrl}^{spec}.$

2) A set $\mathcal{I}$ is called minimal rank-efficient intervention if the
following assertions hold simultaneously:

(i) one has $\left\Vert B\left(  \mathcal{I}\right)  \right\Vert
_{ctrl}^{rank}=m;$

(ii) for every $\mathcal{J\subset}\left\{  1,...r\right\}  $ such that
$\left\vert \mathcal{J}\right\vert <\left\vert \mathcal{I}\right\vert ,$ one
has $\left\Vert B\left(  \mathcal{J}\right)  \right\Vert _{ctrl}^{rank}<m.$
\end{definition}

The reader is invited to note that any minimal spectral-efficient intervention
is also minimal rank-efficient. A condition of type (iii) has no meaning for
the rank, being trivially satisfied as soon as $\mathcal{I}$ is minimal
rank-efficient$.$

To find minimal efficient intervention, one has to solve at most $r$
set-function optimization problems of type%
\[
\max_{\substack{\mathcal{I\subset}\left\{  1,...r\right\}  \\\left\vert
\mathcal{I}\right\vert =k}}\left\Vert B\left(  \mathcal{I}\right)  \right\Vert
_{ctrl},\text{ }1\leq k\leq r,
\]
where $\left\Vert \cdot\right\Vert _{ctrl}$ denotes either $\left\Vert
\cdot\right\Vert _{ctrl}^{spec}$ or $\left\Vert \cdot\right\Vert
_{ctrl}^{rank}.$ It is obvious that
\[
\min\left\{  k:1\leq k\leq r,\max_{\substack{\mathcal{I\subset}\left\{
1,...r\right\}  \\\left\vert \mathcal{I}\right\vert =k}}\left\Vert B\left(
\mathcal{I}\right)  \right\Vert _{ctrl}^{rank}=m\right\}  =\min\left\{
k:1\leq k\leq r,\max_{\substack{\mathcal{I\subset}\left\{  1,...r\right\}
\\\left\vert \mathcal{I}\right\vert =k}}\left\Vert B\left(  \mathcal{I}%
\right)  \right\Vert _{ctrl}^{spec}>0\right\}  .
\]
At this point, one may note that working with minimal spectral-efficient
interventions gives more information and may wonder why we have introduced the
two concept. It turns out that, although both set functions are
non-decreasing, rank-based functions have another useful (submodularity)
property (cf. \cite{NemhauserWolseyFisher_78}, \cite{Lovasz:1983aa}; see also
\cite{SummersCortesiLygeros_2014}). Let us recall the definition of this concept.

\begin{definition}
Given a finite set $S$, a real-valued function $f:2^{S}\longrightarrow%
%TCIMACRO{\U{211d} }%
%BeginExpansion
\mathbb{R}
%EndExpansion
$ is said to be submodular if%
\[
f\left(  S_{1}\cap S_{2}\right)  +f\left(  S_{1}\cup S_{2}\right)  \leq
f\left(  S_{1}\right)  +f\left(  S_{2}\right)  ,
\]
for all subsets $S_{1},S_{2}\subset S.$
\end{definition}

According to \cite{NemhauserWolseyFisher_78}, submodularity is "a
combinatorial analogue of concavity" in the sense that if the cost functional
is submodular, although the problem is NP-hard, a greedy approach provides
good results. The paper \cite[Section 4]{NemhauserWolseyFisher_78} equally
provides greedy heuristics as well as relative error bounds concerning the
greedy solution.

A glance at \cite[Example 1.2]{Lovasz:1983aa} shows that rank-based set
functions are submodular. It turns out that, although it provides more
information, $\left\Vert \cdot\right\Vert _{ctrl}^{spec}$ does not, in all
generality, provide a submodular application. For an example in this
direction, the reader is referred to \cite[III.F]{SummersCortesiLygeros_2014}.

To sum up the considerations made so far, one should begin with solving the
optimization problem using $\left\Vert \cdot\right\Vert _{ctrl}^{rank}$ which
is faster (using greedy heuristic as in \cite[Section 4]%
{NemhauserWolseyFisher_78})$.$ This will provide a minimal $k$ for which
efficient interventions exist. Then, for this particular $k$, one may work
with $\left\Vert \cdot\right\Vert _{ctrl}^{spec}.$

\subsection{Hasty et al.-Inspired Toy Model}

In systems biology, one uses various methods (graphs in
\cite{GaySolimanFages2010}, tropical aspects \cite{SolimanFagesRadulescu2014},
differential equations \cite{MadelaineLhoussaineNiehren}, stochastic analysis,
etc.) to analyse and reduce the same system. The construction is more or less
automatic given the system of reactions. Moreover the qualitative properties
are implicitly thought to be the same in any type of (coherent) model (ODE,
pure jumps, stochastic hybrid, etc.).

A reversible equation $\left(  A+B\overset{K}{\rightleftarrows}C+D\right)  $
allows to obtain\ the ratio of the constants ($K=\left(  k_{+},k_{-}\right)
$) by the so-called \textquotedblleft law of mass action\textquotedblright%
\ going back to the considerations of \cite{C.M.GuldbergandP.Waage}.\ These
hints are valid \textquotedblleft at equilibrium\textquotedblright\ when the
state is invariant. In other words, in stochastic semantics, this is valid as
ergodic behavior/ invariant law, etc. But then, by going over the simplest
examples (and this is equally confirmed by discussions with researchers in
bioinformatics), we learn that the actual values of the reaction constants are
taken from different tables available in the literature and using various
\textquotedblleft normalizations\textquotedblright\ (see, for example,
\cite{crudu_debussche_radulescu_09}). As consequence, even for the simplest
model (Cook, cf. \cite{cook_gerber_tapscott_98}), different values of
parameters (e.g. half life-time) lead to different behaviors (slow unstable,
fast unstable, stable, etc. ). From our point of view, it seems fair to assume
that changing external factors induces changes in the dynamics and this
translates in control.

\textbf{Description} We start with a toy example concerning the temperate
$\lambda$ virus. The phage model is considered to be the standard element in
recombineering (through its Red operon) in order to get a targeted response.
Now, if we take the model developed in \cite{hasty_pradines_dolnik_collins_00}%
, the difference between entering lytic phase or remaining in lysogenic one is
linked to the possibility of keeping CI repressor away from or onto 0. So, in
order for the procedure to be interesting, one has to be able to drive CI
repressor to given targets (corresponding to lytic stage).

The authors of \cite{hasty_pradines_dolnik_collins_00} propose a genetic
applet consisting in a mutant system in which two operator sites (OR2 and OR3)
are present. The gene cI expresses repressor (CI), which dimerizes and binds
to the DNA as a transcription factor in one of the two available sites. The
site OR2 leads to enhanced transcription, while OR3 represses transcription.
We let $R_{1}$ stand for the repressor, $R_{2}$ for the dimer, $D$ for the DNA
promoter site, $DR_{2}$ for the binding to the OR2 site, $DR_{2}^{\ast}$ for
the binding to the OR3 site and $DR_{2}R_{2}$ for the binding to both sites.
We also denote by $P$ the RNA\ polymerase concentration and by $r$ the number
of repressors per mRNA transcript. The capital letters $K_{i},$ $1\leq i\leq4$
for the reversible reactions correspond to couples of direct/reverse speed
functions $k_{i},k_{-i},$ while $K_{t}$ and $K_{d}$ only to direct speed
functions $k_{t}$ and $k_{d}$.The actual system of biochemical reactions that
govern the genetic applet is given by%
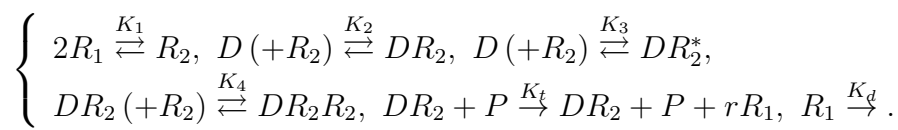
\[
\left\{
\begin{array}
[c]{l}%
2R_{1}\overset{K_{1}}{\rightleftarrows}R_{2},\text{ }D\left(  +R_{2}\right)
\overset{K_{2}}{\rightleftarrows}DR_{2},\text{ }D\left(  +R_{2}\right)
\overset{K_{3}}{\rightleftarrows}DR_{2}^{\ast},\\
DR_{2}\left(  +R_{2}\right)  \overset{K_{4}}{\rightleftarrows}DR_{2}%
R_{2},\text{ }DR_{2}+P\overset{K_{t}}{\rightarrow}DR_{2}+P+rR_{1},\text{
}R_{1}\overset{K_{d}}{\rightarrow}.
\end{array}
\right.
\]

From the stochastic point of view, one can either take Gillespie's method
(pure jumps, cf. \cite{Gillespie1977}) or average on some components (getting
a PDMP mechanism, cf. \cite{crudu_debussche_radulescu_09},
\cite{crudu_Debussche_Muller_Radulescu_2012}). The two types of models are
said to be coherent with each-other if the steady distributions are close (see
\cite{crudu_debussche_radulescu_09}, page 15, Figure I, b). This is, again,
natural, because ergodic behavior of the PDMP often relates to that of the
natural associated Markov chain (e.g. \cite{Costa_Dufour_08}).

But, since these systems (discrete/continuous) are \textquotedblleft
equivalent\textquotedblright\ in the acceptation of bioinformatics, entering
lytic stage for one model of phage lambda or the other should also be
equivalent. This is why we have proposed this paper going along the path of
\cite{GoreacMartinez2015}.

Our concerns are : \ Can one characterize the targeted behavior (by adapting
variable parameters) ? If explicit conditions can be given, are these
"universal" (independent of the way of modelling) ? What are the minimal
parameters on which one must play in order to get the desired targeted
behavior ?

\textbf{The Trend Component }$L$ We consider the trend component given by the
DNA mechanism of the host E-Coli
\[
\left(  D,DR_{2},DR_{2}^{\ast},DR_{2}R_{2}\right)  ^{T}\text{ which belongs to
the basis }\mathcal{B\subset}%
%TCIMACRO{\U{211d} }%
%BeginExpansion
\mathbb{R}
%EndExpansion
^{4}.
\]
All the reactions concerning at least one of these components is considered to
belong to the trend mechanism. The remaining reactions
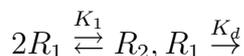
\[
2R_{1}\overset{K_{1}}{\rightleftarrows}R_{2},R_{1}\overset{K_{d}}{\rightarrow}%
\]
will be employed to describe the repressor's updating. To simplify the
arguments (recall that this is a toy model), we consider that all the speeds
in the trend mechanism are unitary ($k_{\pm2}=k_{\pm3}=k_{\pm4}=k_{t}=1$).
Whenever the system is at position $e_{1}$ (unoccupied host DNA), two
reactions are possible $D\overset{k_{2}}{\rightarrow}DR_{2}$, respectively
$D\overset{k_{3}}{\rightarrow}DR_{2}^{\ast}$. We consider that transition
probability is proportional to the speed of reaction (similar to
\cite{Gillespie1977}) to get
\[
\mathbb{P}\left(  L_{n+1}=e_{2}/L_{n}=e_{1}\right)  =\frac{k_{2}}{k_{2}+k_{3}%
},\text{ respectively }\mathbb{P}\left(  L_{n+1}=e_{3}/L_{n}=e_{1}\right)
=\frac{k_{3}}{k_{2}+k_{3}}.
\]
Similar constructions are true for the remaining transitions. Obviously, this
does not correspond to the independent framework since the transition matrix
\[
\mathbb{Q}^{0}\mathbb{=}\left(
\begin{array}
[c]{cccc}%
0 & \frac{k_{2}}{k_{2}+k_{3}} & \frac{k_{3}}{k_{2}+k_{3}} & 0\\
\frac{k_{-2}}{k_{-2}+k_{4}} & 0 & 0 & \frac{k_{4}}{k_{-2}+k_{4}}\\
1 & 0 & 0 & 0\\
0 & 1 & 0 & 0
\end{array}
\right)  =\left(
\begin{array}
[c]{cccc}%
0 & \frac{1}{2} & \frac{1}{2} & 0\\
\frac{1}{2} & 0 & 0 & \frac{1}{2}\\
1 & 0 & 0 & 0\\
0 & 1 & 0 & 0
\end{array}
\right)  .
\]
Nevertheless, we shall assume that the host is at equilibrium prior to
infection and it is easy to see that the unique invariant measure is the
uniform law given by
\begin{equation}
q_{1}=q_{2}=q_{3}=q_{4}=\frac{1}{4}. \label{qHasty}%
\end{equation}

\textbf{The updating matrices }$A_{n}$ To the transitions $2R_{1}%
\overset{K_{1}}{\rightleftarrows}R_{2}$ and $R_{1}\overset{k_{d}}{\rightarrow
}$ one usually associates the ordinary differential equations
\[
\frac{dr_{1}}{dt}=-2k_{1}r_{1}^{2}-k_{d}r_{1}+2k_{-1}r_{2},\text{ }%
\frac{dr_{2}}{dt}=k_{1}r_{1}^{2}-k_{-1}r_{2}.
\]
By writing down the associated Jacobian matrix at some point $r^{0}=\left(
r_{1}^{0},r_{2}^{0}\right)  $, one gets a first-order approximation $\Delta
r=\left(
\begin{array}
[c]{cc}%
-4k_{1}r_{1}^{0}-k_{d} & 2k_{-1}\\
2k_{1}r_{1}^{0} & -k_{-1}%
\end{array}
\right)  r.$ In other words, one constructs the matrix
\[
A=I_{2\times2}+\Delta r=\left(
\begin{array}
[c]{cc}%
1-4k_{1}r_{1}^{0}-k_{d} & 2k_{-1}\\
2k_{1}r_{1}^{0} & 1-k_{-1}%
\end{array}
\right)  .
\]
If $r_{1}^{0}=0,$ then the updating of the dimer is done independently of the
repressor which is not very realistic. For our toy model, we consider
$2k_{1}r_{1}^{0}=k_{d}=k_{-1}=\frac{1}{4}$ i.e.
\begin{equation}
A=\left(
\begin{array}
[c]{cc}%
\frac{1}{4} & \frac{1}{2}\\
\frac{1}{4} & \frac{3}{4}%
\end{array}
\right)  . \label{AHasty}%
\end{equation}
In this framework, $25\%$ of the repressor molecules are degraded ($k_{d}$),
$50\%$ (i.e $4k_{1}r_{1}^{0}$) bind together to produce a total of
$\frac{4k_{1}r_{1}^{0}}{2}r_{1}$ dimers and $25\%$ remain unaltered. For the
dimer, $25\%$ (i.e. $k_{-1}$) break to produce $2k_{-1}r_{2}$ repressors and
$75\%$ remain unaltered.

\begin{remark}
Another way of defining $A_{n}$ would be to keep for $\Delta r$ the actual
Jacobian evaluated at the expectation of uncontrolled $X_{n}$ i.e.
$A_{n}=\left(
\begin{array}
[c]{cc}%
1-4k_{1}\mathbb{E}\left[  X_{n}^{1}\right]  -k_{d} & 2k_{-1}\\
2k_{1}\mathbb{E}\left[  X_{n}^{1}\right]  & 1-k_{-1}%
\end{array}
\right)  $, then compute $\left(
\begin{array}
[c]{c}%
\mathbb{E}\left[  X_{n}^{1}\right] \\
\mathbb{E}\left[  X_{n}^{2}\right]
\end{array}
\right)  =A_{n-1}\left(
\begin{array}
[c]{c}%
\mathbb{E}\left[  X_{n-1}^{1}\right] \\
\mathbb{E}\left[  X_{n-1}^{2}\right]
\end{array}
\right)  ,$ etc.
\end{remark}

\textbf{The Multiplicative Noise }Changes in the trend component have an
effect on the couple repressor/dimer in the transcription phase $DR_{2}+P$
\ $\overset{K_{t}}{\rightarrow}DR_{2}+P+rR_{1}.$ A careful look at the
biochemical reactions shows that binding to the promoter site needs a dimer
per binding. Since the DNA mechanism is assumed to be at equilibrium, the
number of "averaged" occupied promoter sites can be set proportional to
$R_{2}$. This reaction will result in a production of $r$ copies per existing
dimer as soon as the trend is set to $e_{2}.$ This implies that
\begin{equation}
C_{2,n}=\left(
\begin{array}
[c]{cc}%
0 & r\\
0 & 0
\end{array}
\right)  .
\end{equation}
The remaining states induce no noise i.e.
\begin{equation}
C_{i,n}=0_{2\times2},\text{ for }i\in\left\{  1,3,4\right\}  .
\end{equation}
Again, in to simplify the arguments, we assume $r=1.$ We also drop the
dependency on $n.$

We deal with a scaled repressor/dimer component and this is why we add these
as pure jump zero-mean contributions. Alternate models are available (see, for
example \cite{Goreac_SIAM_2015}).

\textbf{One vs. Multi-dimensional controls. }At this point, we envisage four
scenarios concerning the couple repressor/dimer : no external control, control
only the dimer, (same one-dimensional) control on both the dimer and repressor
or control (two-dimensional) on each state. To control the dimer, respectively
mutually control the couple repressor/dimer, one uses $b_{1}=\left(
\begin{array}
[c]{c}%
0\\
1
\end{array}
\right)  ,$ respectively $b_{2}=\left(
\begin{array}
[c]{c}%
1\\
1
\end{array}
\right)  .$ To control the two components independently, one uses $\left[
b_{1},b_{2}\right]  \in%
%TCIMACRO{\U{211d} }%
%BeginExpansion
\mathbb{R}
%EndExpansion
^{2\times2}$ (which is equivalent, up to renaming the controls, to the use of
$B=I_{2\times2}$). Note that these scenarios correspond to selecting a subset
of $\left\{  1,2\right\}  $.

When $B=b_{1},$ one computes $\left(  A^{T}\right)  ^{-1}=\left(
\begin{array}
[c]{cc}%
12 & -4\\
-8 & 4
\end{array}
\right)  $ and $\left(  A^{T}\right)  ^{-1}C_{2}^{T}=\allowbreak\left(
\begin{array}
[c]{cc}%
-4 & 0\\
4 & 0
\end{array}
\right)  $ and notes that $\ker B^{T}$ is $\left(  A^{T}\right)
^{-1}+2\left(  A^{T}\right)  ^{-1}C_{2}^{T}$-invariant and, thus, $\left(
\left(  A^{T}\right)  ^{-1};\left(  A^{T}\right)  ^{-1}\mathcal{C}^{T}\right)
-$strictly invariant (with the notations of Proposition
\ref{PropNecInvariance}). Therefore, the system is not null-controllable
according to Proposition \ref{PropNecInvariance}.

When $B=b_{2},$ we compute the explicit solution of (\ref{MDP2}) associated to
a particular choice of the control as follows.\ For every $x=\left(
\begin{array}
[c]{c}%
x^{1}\\
x^{2}%
\end{array}
\right)  \in%
%TCIMACRO{\U{211d} }%
%BeginExpansion
\mathbb{R}
%EndExpansion
^{2}$, we set
\[
u_{1}=-\frac{1}{4}x^{1}-\frac{3}{4}x^{2}\text{ and }u_{2}=-\frac{1}{4}\left(
\left\langle L_{1},e_{2}\right\rangle -\frac{1}{2}\right)  x^{2},
\]
to get
\[
X_{1}^{x,u}=\left(
\begin{array}
[c]{c}%
\left(  \left\langle L_{1},e_{2}\right\rangle -\frac{1}{2}\right)  x^{2}\\
0
\end{array}
\right)  ,\text{ }X_{2}^{x,u}=0_{2\times1}.
\]
Then, the system is null-controllable.

\begin{remark}
Alternatively, one can use Proposition \ref{PropBSRDSNonRandom} and compute
$\underset{\varepsilon\rightarrow0+}{\lim\inf}P_{0}^{\varepsilon}$ starting
from $P_{2}^{\varepsilon}=0_{2\times2}.$ For small values of $\varepsilon
\simeq10^{-10}$, numerical values stabilize around $\left(
\begin{array}
[c]{cc}%
592 & -192\\
-192 & 64
\end{array}
\right)  $ which is positive-definite. Its minimal eigenvalue is $\left\Vert
B\left(  \left\{  2\right\}  \right)  \right\Vert _{ctrl}^{spec}=\left\Vert
b_{2}\right\Vert _{ctrl}^{spec}\simeq1.5647078.$
\end{remark}

The system is also controllable if $B=\left[  b_{1},b_{2}\right]  .$ In view
of Definition \ref{DefMEI}, it follows that $\mathcal{I=}\left\{  2\right\}  $
provides a minimal efficient intervention plan. In this case, rank and
spectral control norms provide the same (unique) answer.

\textbf{Conclusion and future work}

One is able to characterize the null controllability of discrete-time systems
associated to biochemical reactions through an explicit controllability metric
given by a backward stochastic Riccati scheme (\ref{RiccatiGeneralIt}) through
Theorem \ref{TheoremRiccati}. In particular, if the biological imperative asks
to be able to drive a specific protein to 0, a discrete-time model in which
$\underset{\varepsilon\rightarrow0+}{\lim\inf}P_{0}^{\varepsilon}$ is not
positive definite is not correct. This metric is specific to the discrete-time
models (equivalence between approximate null and exact null controllability in
continuous-time is less obvious).

The explicit algebraic conditions on the coefficients are not universal.
Achieving controllability for continuous-time models does not guarantee
similar behavior for discrete ones (cf. Section \ref{Section2.6}).
Furthermore, while in continuous-time controlling systems with non-random
coefficients to zero guarantees approximate controllability to any random
target, this is no longer valid for the discrete-time model (cf. Example
\ref{ExpNullCtrlNotAppCtrl}). Thus, even though the ergodic behavior of the
discrete or PDMP models lead to the same reaction speeds, choosing one model
or another leads to qualitative differences (in getting lytic behavior and,
hence, recombination).

Finally, if several parameters can be altered, the controllability metric
provides a minimal-intervention scenario (minimal actions on the reactions).
In particular, knowing this scenario, reduces the manipulation costs (reaction
speed depends on adjuvants, temperature, catalyzers, etc.)

Several open questions remain. From a theoretical point of view, we work on
getting explicit algebraic conditions leading to approximate controllability
(not only towards 0) in the general multiplicative models with random
coefficients (both in continuous and discrete framework). In the continuous
framework, some general sufficient condition can obtained by using the
explicit reduction of BSDE to systems of ODE (from \cite{CFJ_2014}) in the
same spirit as the proof of Proposition \ref{PropNecInvariance}. At
application level, similar procedures can be envisaged for objective-based
systems reduction. In this case, the decision is made at subgraph selection
level : what reactions to be suppressed and what reactions to be added to
preserve a given property. The aim in this framework is to give the smallest
set of reactions allowing to achieve the goal. This makes the object of
on-going research in both discrete and continuous framework. The method could
combine the simplification set of rules in \cite{MadelaineLhoussaineNiehren}
for the deterministic case and the subgraph selection in
\cite{SummersCortesiLygeros_2014} (but according to the stochastic metric
introduced in the present paper).

\section{Proofs of Theorems \ref{PropNSC} and \ref{TheoremRiccati},
Solvability Propositions \ref{PropBSRDSNonRandom} and \ref{PropBSRDSC=0} and
Necessary Condition (Proposition \ref{PropNecInvariance})\label{Section4}}

\subsection{Proof of the Main Results}

We begin with the proof of the duality-based characterization of
controllability concepts.

\begin{proof}
[Proof of Theorem \ref{PropNSC}]The first two assertions are direct
consequences of the duality properties. One easily notes that
\begin{align*}
&  \mathbb{E}\left[  \left\langle X_{n+1}^{x,u},Y_{n+1}^{N,\xi}\right\rangle
/\mathcal{F}_{n}\right] \\
&  =\left\langle A_{n}X_{n}^{x,u}+Bu_{n+1},\mathbb{E}\left[  Y_{n+1}^{N,\xi
}/\mathcal{F}_{n}\right]  \right\rangle +\left\langle X_{n}^{x,u}%
,\mathbb{E}\left[  \sum_{i=1}^{p}\left\langle \Delta M_{n+1},e_{i}%
\right\rangle C_{i,n}^{T}Z_{n}^{N,\xi}\Delta M_{n+1}/\mathcal{F}_{n}\right]
\right\rangle \\
&  =\left\langle X_{n}^{x,u},A_{n}^{T}\mathbb{E}\left[  Y_{n+1}^{N,\xi
}/\mathcal{F}_{n}\right]  +\sum_{i=1}^{p}C_{i,n}^{T}Z_{n}^{N,\xi}%
\mathbb{E}\left[  \left\langle \Delta M_{n+1},e_{i}\right\rangle \Delta
M_{n+1}/\mathcal{F}_{n}\right]  \right\rangle +\left\langle Bu_{n+1}%
,\mathbb{E}\left[  Y_{n+1}^{N,\xi}/\mathcal{F}_{n}\right]  \right\rangle \\
&  =\left\langle X_{n}^{x,u},Y_{n}^{N,\xi}\right\rangle +\left\langle
Bu_{n+1},\mathbb{E}\left[  Y_{n+1}^{N,\xi}/\mathcal{F}_{n}\right]
\right\rangle .
\end{align*}
Hence, by iterating, one gets
\begin{equation}
\mathbb{E}\left[  \left\langle X_{N}^{x,u},Y_{N}^{N,\xi}\right\rangle \right]
=\left\langle x,Y_{0}^{N,\xi}\right\rangle +%
%TCIMACRO{\tsum \limits_{n=0}^{N-1}}%
%BeginExpansion
{\textstyle\sum\limits_{n=0}^{N-1}}
%EndExpansion
\mathbb{E}\left[  \left\langle u_{n+1},B^{T}\mathbb{E}\left[  Y_{n+1}^{N,\xi
}/\mathcal{F}_{n}\right]  \right\rangle \right]  . \label{duality}%
\end{equation}
One proceeds in a classical manner by considering the two linear operators
\begin{align*}
R_{N}^{1}  &  :\mathcal{P}red\mathcal{\longrightarrow}\mathbb{L}^{2}\left(
\Omega,\mathcal{F}_{N},\mathbb{P};%
%TCIMACRO{\U{211d} }%
%BeginExpansion
\mathbb{R}
%EndExpansion
^{m}\right)  ,\text{ }R_{N}^{1}\left(  u\right)  =X_{N}^{0,u},\text{ for all
}u\in\mathcal{P}red,\\
R_{N}^{2}  &  :%
%TCIMACRO{\U{211d} }%
%BeginExpansion
\mathbb{R}
%EndExpansion
^{m}\mathcal{\longrightarrow}\mathbb{L}^{2}\left(  \Omega,\mathcal{F}%
_{N},\mathbb{P};%
%TCIMACRO{\U{211d} }%
%BeginExpansion
\mathbb{R}
%EndExpansion
^{m}\right)  ,\text{ }R_{N}^{2}\left(  x\right)  =X_{N}^{x,0},\text{ for all
}x\in\mathbb{R}^{m}.
\end{align*}
The reader is invited to recall that $\mathcal{P}red$ stands for the family of
$%
%TCIMACRO{\U{211d} }%
%BeginExpansion
\mathbb{R}
%EndExpansion
^{d}$-valued, $\mathbb{F}$-predictable controls. (It is considered here as a
subspace of product of $\mathbb{L}^{2}\left(  \Omega,\mathcal{F}%
_{n},\mathbb{P};%
%TCIMACRO{\U{211d} }%
%BeginExpansion
\mathbb{R}
%EndExpansion
^{d}\right)  $-spaces.) In view of (\ref{duality}), the adjoints of the linear
operators are given by
\begin{equation}
\left(  R_{N}^{1}\right)  ^{\ast}\left(  \xi\right)  =\left(  B^{T}%
\mathbb{E}\left[  Y_{n}^{N,\xi}/\mathcal{F}_{n-1}\right]  \right)  _{n\geq
1},\text{ }\left(  R_{N}^{2}\right)  ^{\ast}\left(  \xi\right)  =Y_{0}^{N,\xi
}, \label{adjoints}%
\end{equation}
for all $\xi\in\mathbb{L}^{2}\left(  \Omega,\mathcal{F}_{N},\mathbb{P};%
%TCIMACRO{\U{211d} }%
%BeginExpansion
\mathbb{R}
%EndExpansion
^{m}\right)  .$ Then, approximate null-controllability is equivalent to the
image (range) inclusion $\operatorname{Im}\left(  R_{N}^{2}\right)  \subset
cl\left(  \operatorname{Im}\left(  R_{N}^{1}\right)  \right)  $, where $cl$ is
the usual Kuratowski closure operator. Furthermore, this is equivalent to the
kernel inclusion $\ker\left(  \left(  R_{N}^{1}\right)  ^{\ast}\right)
\subset\ker\left(  \left(  R_{N}^{2}\right)  ^{\ast}\right)  $ which leads to
the second assertion. Similar, approximate controllability is equivalent to
$cl\left(  \operatorname{Im}\left(  R_{N}^{1}\right)  \right)  =\mathbb{L}%
^{2}\left(  \Omega,\mathcal{F}_{N},\mathbb{P};%
%TCIMACRO{\U{211d} }%
%BeginExpansion
\mathbb{R}
%EndExpansion
^{m}\right)  .$ Hence, it is equivalent to $\ker\left(  \left(  R_{N}%
^{1}\right)  ^{\ast}\right)  =\left\{  0\right\}  $ which leads to the first assertion.

For the third assertion, since $\Omega$ is assumed to be the sample space, it
follows that $\mathbb{L}^{2}\left(  \Omega,\mathcal{F}_{N},\mathbb{P};%
%TCIMACRO{\U{211d} }%
%BeginExpansion
\mathbb{R}
%EndExpansion
^{m}\right)  $ can be seen as $%
%TCIMACRO{\U{211d} }%
%BeginExpansion
\mathbb{R}
%EndExpansion
^{mp^{N}}$. Hence, the linear subspace $\operatorname{Im}\left(  R_{N}%
^{1}\right)  $ is finite-dimensional and, thus, $cl\left(  \operatorname{Im}%
\left(  R_{N}^{1}\right)  \right)  =\operatorname{Im}\left(  R_{N}^{1}\right)
$. In this case, approximate null-controllability is written down as
$\operatorname{Im}\left(  R_{N}^{2}\right)  \subset\operatorname{Im}\left(
R_{N}^{1}\right)  $ (i.e. one actually has exact null-controllability), or,
equivalently (see, for example, \cite[Appendix B, Proposition B.1]%
{DaPratoZabczyk1992}),
\[
\left\vert \left(  R_{N}^{2}\right)  ^{\ast}\xi\right\vert \leq k\left\Vert
\left(  R_{N}^{1}\right)  ^{\ast}\xi\right\Vert _{\mathcal{P}red},\text{ for
some }k>0.
\]
The necessary and sufficient condition (\ref{0ctrlInequality}) follows from
(\ref{adjoints}).
\end{proof}

We now give the proof of the second main result of the paper providing the
link between the controllability (pseudo)norm and the backward stochastic
Riccati difference scheme.

\begin{proof}
[Proof of Theorem \ref{TheoremRiccati}]For the first assertion, using the
particular form of $\alpha_{n,\varepsilon}$ and $\eta_{n,\varepsilon}$, one
simply writes down%
\begin{align*}
&  \left\langle P_{n}^{\varepsilon}y_{n}^{y_{0},v},y_{n}^{y_{0},v}%
\right\rangle \\
&  =\mathbb{E}\left[  \left\langle P_{n+1}^{\varepsilon}\left[  A_{n}%
^{T}\right]  ^{-1}y_{n}^{y_{0},v},\left[  A_{n}^{T}\right]  ^{-1}y_{n}%
^{y_{0},v}\right\rangle /\mathcal{F}_{n}\right]  +\left\vert B^{T}\left[
A_{n}^{T}\right]  ^{-1}y_{n}^{y_{0},v}\right\vert ^{2}-\left\langle
\alpha_{n,\varepsilon}^{T}\eta_{n,\varepsilon}^{-1}\alpha_{n,\varepsilon}%
y_{n}^{y_{0},v},y_{n}^{y_{0},v}\right\rangle \\
&  =\mathbb{E}\left[  \left\langle P_{n+1}^{\varepsilon}y_{n+1}^{y_{0}%
,v},y_{n+1}^{y_{0},v}\right\rangle /\mathcal{F}_{n}\right]  +\varepsilon
\left\vert v_{n+1}\right\vert ^{2}+\left\vert B^{T}\mathbb{E}\left[
y_{n+1}^{y_{0},v}/\mathcal{F}_{n}\right]  \right\vert ^{2}-\left\vert
\eta_{n,\varepsilon}^{-1/2}\alpha_{n,\varepsilon}y_{n}^{y_{0},v}%
-\eta_{n,\varepsilon}^{1/2}v_{n+1}\right\vert ^{2}.
\end{align*}
By iterating and taking expectation, one gets%
\begin{align}
\left\langle P_{0}^{\varepsilon}y_{0},y_{0}\right\rangle  &  =\varepsilon
\sum_{n=0}^{N-1}\mathbb{E}\left[  \left\vert v_{n+1}\right\vert ^{2}\right]
+\mathbb{E}\left[  \sum_{n=0}^{N-1}\left\vert B^{T}\mathbb{E}\left[
y_{n+1}^{y_{0},v}/\mathcal{F}_{n}\right]  \right\vert ^{2}\right] \nonumber\\
&  -\sum_{n=0}^{N-1}\mathbb{E}\left[  \left\vert \eta_{n,\varepsilon}%
^{-1/2}\delta_{n,\varepsilon}y_{n}^{y_{0},v}-\eta_{n,\varepsilon}^{1/2}%
v_{n+1}\right\vert ^{2}\right]  . \label{optimalityPeps}%
\end{align}
If the system (\ref{MDP2}) is (approximately) null-controllable, then there
exists some positive constant $c>0$ such that
\[
\inf_{\left(  v_{n}\right)  _{1\leq n\leq N}\text{ }\mathbb{F}%
\text{-predictable}}\mathbb{E}\left[  \sum_{n=0}^{N-1}\left\vert
B^{T}\mathbb{E}\left[  y_{n+1}^{y_{0},v}/\mathcal{F}_{n}\right]  \right\vert
^{2}\right]  \geq c\left\vert y_{0}\right\vert ^{2}.
\]
In particular, by taking the feedback control $v_{n+1}^{\varepsilon}%
:=\eta_{n,\varepsilon}^{-1}\delta_{n,\varepsilon}y_{n}^{y_{0},v^{\varepsilon}%
},$ one establishes that
\[
\left\langle P_{0}^{\varepsilon}y_{0},y_{0}\right\rangle \geq c\left\vert
y_{0}\right\vert ^{2}%
\]
and the conclusion follows.

Conversely, if $\underset{\varepsilon\rightarrow0+}{\lim\inf}P_{0}%
^{\varepsilon}\geq cI,$ for some $c>0,$ then, for every $\varepsilon>0$ small
enough and every predictable control $v$, one gets
\[
\varepsilon\sum_{n=0}^{N-1}\mathbb{E}\left[  \left\vert v_{n+1}\right\vert
^{2}\right]  +\mathbb{E}\left[  \sum_{n=0}^{N-1}\left\vert B^{T}%
\mathbb{E}\left[  y_{n+1}^{y_{0},v}/\mathcal{F}_{n}\right]  \right\vert
^{2}\right]  \geq\frac{c}{2}\left\vert y_{0}\right\vert ^{2}%
\]
and the conclusion follows by letting $\varepsilon\rightarrow0$.

For the second assertion, one notes that (\ref{optimalityPeps}) implies that
\[
\left\langle P_{0}^{\varepsilon}y_{0},y_{0}\right\rangle =\inf_{\left(
v_{n}\right)  _{1\leq n\leq N}\text{ }\mathbb{F}\text{-predictable}}\left(
\varepsilon\sum_{n=0}^{N-1}\mathbb{E}\left[  \left\vert v_{n+1}\right\vert
^{2}\right]  +\mathbb{E}\left[  \sum_{n=0}^{N-1}\left\vert B^{T}%
\mathbb{E}\left[  y_{n+1}^{y_{0},v}/\mathcal{F}_{n}\right]  \right\vert
^{2}\right]  \right)
\]
and the conclusion follows by letting $\varepsilon\rightarrow0$.
\end{proof}

\subsection{Proof of the Solvability Results}

We begin with the proof for the solvability of the BSRDS in the case of
non-random coefficients.

\begin{proof}
[Proof of Proposition \ref{PropBSRDSNonRandom}]The proof is given by
(descending) induction. For $n=N,$ it is clear that $Q_{N-1}^{\varepsilon}=0$.
Since $A_{N-1}$ and $\mathcal{C}_{N-1}$ are deterministic, it is clear that
the iterative step defining $P_{N-1}^{\varepsilon}$ in scheme
(\ref{RiccatiGeneralIt}) reduces to (\ref{RiccatiNonRandomCoeff}). Let us
assume that $P_{n+1}^{\varepsilon}$ has been constructed according to this
deterministic scheme and is a positive-semidefinite (non-random)matrix. Then,
$Q_{n}^{\varepsilon}=0.$ Since $A_{n}$ and $C_{n}$ are deterministic, the
definition of $P_{n}^{\varepsilon}$ according to scheme
(\ref{RiccatiGeneralIt}) reduces to (\ref{RiccatiNonRandomCoeff}). We only
need to prove that this latter scheme is consistent and provides
positive-semidefinite matrices. One begin with noting that as soon as
$P_{n+1}^{\varepsilon}$ is positive-semidefinite, the matrix
\[
\Delta_{n+1}^{\varepsilon}:=\left(  q_{j}\left(  \delta_{j,k}-q_{k}\right)
P_{n+1}^{\varepsilon}\right)  _{1\leq j,k\leq p}\in%
%TCIMACRO{\U{211d} }%
%BeginExpansion
\mathbb{R}
%EndExpansion
^{mp\times mp}%
\]
is also positive-semidefinite. Indeed, it suffices to set $DD^{T}=\left(
q_{j}\left(  \delta_{j,k}-q_{k}\right)  \right)  _{1\leq j,k\leq p}$ given by
Cholesky decomposition and $\mathcal{D}_{j,k}:=D_{j,k}I_{m\times m},$ for all
$1\leq j,k\leq p.$ Then
\[
\Delta_{n+1}^{\varepsilon}=\mathcal{D}\left(
\begin{array}
[c]{ccccc}%
P_{n+1}^{\varepsilon} & 0_{m\times m} & 0_{m\times m} & ... & 0_{m\times m}\\
0_{m\times m} & P_{n+1}^{\varepsilon} & 0_{m\times m} & ... & 0_{m\times m}\\
0_{m\times m} & 0_{m\times m} & P_{n+1}^{\varepsilon} & ... & 0_{m\times m}\\
... & ... & ... & ... & ...\\
0_{m\times m} & 0_{m\times m} & 0_{m\times m} & ... & P_{n+1}^{\varepsilon}%
\end{array}
\right)  \mathcal{D}^{T}%
\]
is obviously positive-semidefinite. It follows that $\eta_{n,\varepsilon}$
given by (\ref{RiccatiNonRandomCoeff}) is positive-definite. Second, using a
classical intuition on feedback optimal control, one writes
\begin{align*}
&  A_{n}^{-1}\left(  P_{n+1}^{\varepsilon}+BB^{T}\right)  \left[  A_{n}%
^{T}\right]  ^{-1}-\alpha_{n,\varepsilon}^{T}\eta_{n,\varepsilon}^{-1}%
\alpha_{n,\varepsilon}\\
&  =\left[  A_{n}^{-1}-\alpha_{n,\varepsilon}^{T}\eta_{n,\varepsilon}%
^{-1}\mathcal{C}_{n}A_{n}^{-1}\right]  \left(  P_{n+1}^{\varepsilon}%
+BB^{T}\right)  \left[  \left[  A_{n}^{-1}\right]  ^{T}-\left[  A_{n}%
^{-1}\right]  ^{T}\mathcal{C}_{n}\eta_{n,\varepsilon}^{-1}\alpha
_{n,\varepsilon}\right] \\
&  +\alpha_{n,\varepsilon}^{T}\eta_{n,\varepsilon}^{-1}\left(  \varepsilon
I_{mp\times mp}+\Delta_{n+1}^{\varepsilon}\right)  \eta_{n,\varepsilon}%
^{-1}\alpha_{n,\varepsilon}.
\end{align*}
This implies that $P_{n}^{\varepsilon}$ is positive-semidefinite whenever
$P_{n+1}^{\varepsilon}$ is positive-semidefinite and the induction step is complete.
\end{proof}

The second proof concerns the solvability of the BSRDS in the absence of
multiplicative noise (i.e. $C=0$).

\begin{proof}
[Proof of Proposition \ref{PropBSRDSC=0}]One begins with setting%
\[
p_{N}^{\varepsilon}=0\text{ and }q_{N}^{\varepsilon}=0
\]
and notes that $P_{N-1}^{\varepsilon}$ given by (\ref{RiccatiC=0}) satisfies
\[%
\begin{array}
[c]{l}%
P_{N-1}^{\varepsilon}=A_{N-1}^{-1}BB^{T}\left[  A_{N-1}^{T}\right]
^{-1}=A^{-1}\left(  N-1,L_{N-1}\right)  \left(  p_{N}^{\varepsilon}%
+BB^{T}-q_{N}^{\varepsilon}\right)  \left(  A^{-1}\left(  N-1,L_{N-1}\right)
\right)  ^{T}.
\end{array}
\]
Next, one recalls that $Q_{N-2}=\left[  Q_{N-2,1}\text{ }Q_{N-2,2}\text{ ...
}Q_{N-2,m}\right]  ,$ where $Q_{N-2,j}\in%
%TCIMACRO{\U{211d} }%
%BeginExpansion
\mathbb{R}
%EndExpansion
^{m\times p}$. \ One easily computes
\[
\left[  Q_{N-2,1}^{\cdot,l},Q_{N-2,2}^{\cdot,l},...,Q_{N-2,m}^{\cdot
,l}\right]  =A^{-1}\left(  N-1,e_{l}\right)  \left(  p_{N}^{\varepsilon
}+BB^{T}-q_{N}^{\varepsilon}\right)  \left(  A^{-1}\left(  N-1,e_{l}\right)
\right)  ^{T},
\]
for all $1\leq l\leq p$ $.$ Therefore, the conclusion holds true for $n=N-1.$
We proceed by (decreasing) induction and assume the conclusion to hold true
for $n+1$ and prove it for $n\leq N-2.$ One easily notes that,due to the
recurrence assumption, the following equalities hold true for $\alpha$ and
$\eta$ computed as in (\ref{RiccatiC=0}).
\[
\left\{
\begin{array}
[c]{l}%
\alpha_{n,\varepsilon}^{j}=-\overline{\alpha}_{n,\varepsilon}^{j}\left[
A_{n}^{T}\right]  ^{-1},\text{ where }\\
\overline{\alpha}_{n,\varepsilon}^{j}=\left[
%TCIMACRO{\tsum \limits_{l=1}^{p}}%
%BeginExpansion
{\textstyle\sum\limits_{l=1}^{p}}
%EndExpansion
q_{l}\left(  \delta_{j,l}-q_{j}\right)  A^{-1}\left(  n+1,e_{l}\right)
\left(  p_{n+2}^{\varepsilon}+BB^{T}-q_{n+2}^{\varepsilon}\right)  \left(
A^{-1}\left(  n+1,e_{l}\right)  \right)  ^{T}\right]  \text{ and}\\
\eta_{n,\varepsilon}^{j,k}=\varepsilon\delta_{j,k}I_{m\times m}+%
%TCIMACRO{\tsum \limits_{l=1}^{p}}%
%BeginExpansion
{\textstyle\sum\limits_{l=1}^{p}}
%EndExpansion
q_{l}\left(  q_{j}-\delta_{j,l}\right)  \left(  q_{k}-\delta_{k,l}\right)
A^{-1}\left(  n+1,e_{l}\right)  \left(  p_{n+2}^{\varepsilon}+BB^{T}%
-q_{n+2}^{\varepsilon}\right)  \left(  A^{-1}\left(  n+1,e_{l}\right)
\right)  ^{T},
\end{array}
\right.
\]
for all $1\leq j,k\leq p$. We set
\begin{equation}
\left\{
\begin{array}
[c]{l}%
p_{n+1}^{\varepsilon}:=\mathbb{E}\left[  P_{n+1}^{\varepsilon}/\mathcal{F}%
_{n}\right]  =%
%TCIMACRO{\tsum \limits_{l=1}^{p}}%
%BeginExpansion
{\textstyle\sum\limits_{l=1}^{p}}
%EndExpansion
q_{l}A^{-1}\left(  n+1,e_{l}\right)  \left(  p_{n+2}^{\varepsilon}%
+BB^{T}-q_{n+2}^{\varepsilon}\right)  \left(  A^{-1}\left(  n+1,e_{l}\right)
\right)  ^{T}\text{,}\\
q_{n+1}^{\varepsilon}=\overline{\alpha}_{n,\varepsilon}^{T}\eta_{n,\varepsilon
}^{-1}\overline{\alpha}_{n,\varepsilon}.
\end{array}
\right.  \label{pqn+1'}%
\end{equation}
We will see in just one moment that $\eta_{n,\varepsilon}^{-1}$ (hence,
$q_{n+1}^{\varepsilon}$) is consistent. By (\ref{RiccatiC=0}), it follows
that
\[
P_{n}^{\varepsilon}=A_{n}^{-1}\left(  p_{n+1}^{\varepsilon}+BB^{T}%
-q_{n+1}^{\varepsilon}\right)  \left[  A_{n}^{T}\right]  ^{-1}.
\]
For $Q,$ the assertion is obtained as in the first step. We come back to the
quantities $p_{n+1}^{\varepsilon}$ and $q_{n+1}^{\varepsilon}$ given by
(\ref{pqn+1'}) and show that they are well-defined and satisfy (\ref{pngeqn}).
The induction assumption $p_{n+2}^{\varepsilon}\geq q_{n+2}^{\varepsilon}$
implies that $p_{n+1}^{\varepsilon}$ is positive-semidifinite. Second, with
the notations
\begin{align*}
\mathcal{A}  &  :\mathcal{=}\left(  \sqrt{q_{1}}A^{-1}\left(  n+1,e_{1}%
\right)  ,...,\sqrt{q_{p}}A^{-1}\left(  n+1,e_{p}\right)  \right)  \in%
%TCIMACRO{\U{211d} }%
%BeginExpansion
\mathbb{R}
%EndExpansion
^{m\times mp},\\
\mathcal{P}  &  :\mathcal{=}\left(  \delta_{j,k}\left(  p_{n+2}^{\varepsilon
}+BB^{T}-q_{n+2}^{\varepsilon}\right)  \right)  _{1\leq j,k\leq p}\in%
%TCIMACRO{\U{211d} }%
%BeginExpansion
\mathbb{R}
%EndExpansion
^{mp\times mp},\text{ }\\
\mathcal{D}  &  :\mathcal{=}\left(  \sqrt{q_{k}}\left(  \delta_{j,k}%
-q_{j}\right)  A^{-1}\left(  n+1,e_{k}\right)  \right)  _{1\leq j,k\leq p}\in%
%TCIMACRO{\U{211d} }%
%BeginExpansion
\mathbb{R}
%EndExpansion
^{mp\times mp},
\end{align*}
one has
\[
\eta_{n,\varepsilon}=\varepsilon I_{mp\times mp}+\mathcal{DPD}^{T}>0\text{ and
}\overline{\alpha}_{n,\varepsilon}=\mathcal{DPA}^{T}.
\]
For the inequality, we have used the induction assumption $p_{n+2}%
^{\varepsilon}\geq q_{n+2}^{\varepsilon}.$ As consequence, $q_{n+1}%
^{\varepsilon}$ is well-defined and positive-semidefinite. Finally,
\begin{align*}
&  p_{n+1}^{\varepsilon}-q_{n+1}^{\varepsilon}=\mathcal{APA}^{T}%
-\mathcal{APD}^{T}\left(  \varepsilon I_{mp\times mp}+\mathcal{DPD}%
^{T}\right)  ^{-1}\mathcal{DPA}^{T}\\
&  =\left(  \mathcal{A-APD}^{T}\left(  \varepsilon I_{mp\times mp}%
+\mathcal{DPD}^{T}\right)  ^{-1}\mathcal{D}\right)  \mathcal{P}\left(
\mathcal{A-APD}^{T}\left(  \varepsilon I_{mp\times mp}+\mathcal{DPD}%
^{T}\right)  ^{-1}\mathcal{D}\right)  ^{T}\\
&  +\varepsilon\mathcal{APD}^{T}\left(  \varepsilon I_{mp\times mp}%
+\mathcal{DPD}^{T}\right)  ^{-1}\left(  \varepsilon I_{mp\times mp}%
+\mathcal{DPD}^{T}\right)  ^{-1}\mathcal{DPA}^{T},
\end{align*}
which is clearly positive-semidefinite. Our induction step is now complete and
the conclusion follows.
\end{proof}

\subsection{Proof of Proposition \ref{PropNecInvariance}}

\begin{proof}
[Proof of Proposition \ref{PropNecInvariance}]We begin with setting
$\mathcal{V}^{N,N}=\ker\left(  B^{T}\right)  .$ We proceed for $0\leq k<N,$
and denote by $\mathcal{V}^{k,N}$ the largest subspace of $\ker\left(
B^{T}\right)  $ which is $\left(  \left[  A^{T}\right]  ^{-1};\left(
\mathcal{C}\left(  1\right)  A^{-1}\right)  ^{T}\Pi_{\mathcal{V}^{k+1,N}%
},...,\left(  \mathcal{C}\left(  p\right)  A^{-1}\right)  ^{T}\Pi
_{\mathcal{V}^{k+1,N}}\right)  -$ invariant. According to \cite[Theorem
3.2]{Schmidt_Stern_80} (see also \cite[Lemma 4.6]{Curtain_86}), the set
$\mathcal{V}^{k,N}$ is equally $\left(  \left[  A^{T}\right]  ^{-1};\left(
\mathcal{C}\left(  1\right)  A^{-1}\right)  ^{T}\Pi_{\mathcal{V}^{k+1,N}%
},...,\left(  \mathcal{C}\left(  p\right)  A^{-1}\right)  ^{T}\Pi
_{\mathcal{V}^{k+1,N}}\right)  -$feedback invariant. Thus, there exists a
family of linear operators $\left(  F^{k}\left(  l\right)  \right)  _{1\leq
l\leq p}\subset%
%TCIMACRO{\U{211d} }%
%BeginExpansion
\mathbb{R}
%EndExpansion
^{m\times m}$ such that $\mathcal{V}^{k,N}$ is $\left(  \left[  A^{T}\right]
^{-1}+\sum_{l=1}^{p}\left(  \mathcal{C}\left(  l\right)  A^{-1}\right)
^{T}\Pi_{\mathcal{V}^{k+1,N}}F^{k}\left(  l\right)  \right)  $- invariant. We
consider the linear stochastic system
\[
\left\{
\begin{array}
[c]{l}%
x_{n+1}^{y_{0}}:=\left(  \left[  A^{T}\right]  ^{-1}+\sum_{l=1}^{p}\left(
\mathcal{C}\left(  l\right)  A^{-1}\right)  ^{T}\Pi_{\mathcal{V}^{n+1,N}}%
F^{n}\left(  l\right)  \right)  x_{n}^{y_{0}}+\overset{p}{\underset{l=1}{\sum
}}\left\langle \Delta M_{n+1},e_{l}\right\rangle \Pi_{\mathcal{V}^{n+1,N}%
}F^{n}\left(  l\right)  x_{n}^{y_{0}},\\
x_{0}^{y_{0}}=y_{0}\in\mathcal{V}^{0,N}.
\end{array}
\right.
\]
Then $x_{n+1}^{y_{0}}$ coincides with the solution of (\ref{MDPdual})
associated to the feedback control $v^{feedback}\left(  n+1,y\right)  =\left[
\Pi_{\mathcal{V}^{n+1,N}}F^{n}\left(  1\right)  y,...,\Pi_{\mathcal{V}%
^{n+1,N}}F^{n}\left(  p\right)  y\right]  ,$ for all $n\geq0.$ Moreover,
whenever $y_{0}\in\mathcal{V}^{0,N}$, one gets $x_{n+1}^{y_{0}}\in
\mathcal{V}^{n+1,N}$, $\mathbb{P-}a.s.$ for all $n\geq0$. In particular,
recalling that $\mathcal{V}^{n+1,N}$ $\subset\ker\left(  B^{T}\right)  $, it
follows that $B^{T}\mathbb{E}\left[  y_{n+1}^{y_{0},v^{feedback}}%
/\mathcal{F}_{n}\right]  =0$, $\mathbb{P-}a.s.$ for all $n\geq0$. By our
controllability assumption and Criterion \ref{CritNullCtrl}, one deduces
$y_{0}=0$ and our proposition is complete by recalling that $y_{0}%
\in\mathcal{V}^{0,N}$ is arbitrary.
\end{proof}

\bibliographystyle{ieeetr}
\bibliography{bibliografie_17042016}

\def\cprime{$'$}
\begin{thebibliography}{10}

\bibitem{GoreacMartinez2015}
D.~Goreac and M.~Martinez, ``Algebraic invariance conditions in the study of
  approximate (null-) controllability of {M}arkov switch processes,'' {\em
  Math. Control Signals Systems}, vol.~27, no.~4, pp.~551--578, 2015.

\bibitem{CohenElliot2011ProgrProb}
S.~N. Cohen and R.~J. Elliott, ``Backward stochastic difference equations with
  finite states,'' in {\em Stochastic Analysis with Financial Applications}
  (A.~Kohatsu-Higa, N.~Privault, and S.-J. Sheu, eds.), vol.~65 of {\em
  Progress in Probability}, pp.~33--42, Springer Basel, 2011.

\bibitem{Cohen_Elliott_SPA_2010}
S.~N. Cohen and R.~J. Elliott, ``A general theory of finite state backward
  stochastic difference equations,'' {\em Stochastic Processes and their
  Applications}, vol.~120, no.~4, pp.~442 -- 466, 2010.

\bibitem{Buckdahn_Quincampoix_Tessitore_2006}
R.~Buckdahn, M.~Quincampoix, and G.~Tessitore, ``A characterization of
  approximately controllable linear stochastic differential equations,'' in
  {\em Stochastic partial differential equations and applications---{VII}},
  vol.~245 of {\em Lect. Notes Pure Appl. Math.}, pp.~53--60, Chapman \&
  Hall/CRC, Boca Raton, FL, 2006.

\bibitem{G17}
D.~Goreac, ``A {K}alman-type condition for stochastic approximate
  controllability,'' {\em Comptes Rendus Mathematique}, vol.~346, no.~3--4,
  pp.~183 -- 188, 2008.

\bibitem{G10}
D.~Goreac, ``A note on the controllability of jump diffusions with linear
  coefficients,'' {\em IMA Journal of Mathematical Control and Information},
  vol.~29, no.~3, pp.~427--435, 2012.

\bibitem{Hautus}
M.~L.~J. Hautus, ``Controllability and observability conditions of linear
  autonomous systems,'' {\em Nederl. Akad. Wetensch. Proc. Ser. A 72 Indag.
  Math.}, vol.~31, pp.~443--448, 1969.

\bibitem{Schmidt_Stern_80}
E.~J. P.~G. Schmidt and R.~J. Stern, ``Invariance theory for infinite
  dimensional linear control systems,'' {\em {Applied Mathematics and
  Optimization}}, vol.~{6}, no.~{2}, pp.~{113--122}, {1980}.

\bibitem{Curtain_86}
R.~F. Curtain, ``{Invariance concepts in infinite dimensions},'' {\em {SIAM J.
  Control and Optim.}}, vol.~{24}, pp.~{1009--1030}, {SEP} {1986}.

\bibitem{russell_Weiss_1994}
D.~Russell and G.~Weiss, ``A general necessary condition for exact
  observability,'' {\em SIAM Journal on Control and Optimization}, vol.~32,
  no.~1, pp.~1--23, 1994.

\bibitem{Jacob_Zwart_2001}
B.~Jacob and H.~Zwart, ``Exact observability of diagonal systems with a
  finite-dimensional output operator,'' {\em {Systems \& Control Letters}},
  vol.~43, no.~2, pp.~101 -- 109, 2001.

\bibitem{Jacob_Partington_2006}
B.~Jacob and J.~R. Partington, ``On controllability of diagonal systems with
  one-dimensional input space,'' {\em {Systems \& Control Letters}}, vol.~55,
  no.~4, pp.~321 -- 328, 2006.

\bibitem{Pardoux_Peng_90}
E.~Pardoux and S.~G. Peng, ``Adapted solution of a backward stochastic
  differential equation,'' {\em {Syst. Control Lett.}}, vol.~14, no.~1, pp.~55
  -- 61, 1990.

\bibitem{Peng_94}
S.~Peng, ``Backward stochastic differential equation and exact controllability
  of stochastic control systems,'' {\em Progr. Natur. Sci. (English Ed.)},
  vol.~4, pp.~274--284, 1994.

\bibitem{Merton_76}
R.~C. Merton, ``Option pricing when underlying stock returns are
  discontinuous,'' {\em J. Financ. Econ.}, vol.~3, pp.~125-- 144, 1976.

\bibitem{G1}
D.~Goreac, ``Controllability properties of linear mean-field stochastic
  systems,'' {\em Stochastic Analysis and Applications}, vol.~32, no.~02,
  pp.~280--297, 2014.

\bibitem{Fernandez_Cara_Garrido_atienza_99}
E.~Fern{\'a}ndez-Cara, M.~J. Garrido-Atienza, and J.~Real, ``On the approximate
  controllability of a stochastic parabolic equation with a multiplicative
  noise,'' {\em C. R. Acad. Sci. Paris S\'er. I Math.}, vol.~328, no.~8,
  pp.~675--680, 1999.

\bibitem{Sarbu_Tessitore_2001}
M.~Sirbu and G.~Tessitore, ``{Null controllability of an infinite dimensional
  SDE with state- and control-dependent noise},'' {\em {Systems \& Control
  Letters}}, vol.~{44}, pp.~{385--394}, {DEC 14} {2001}.

\bibitem{Barbu_Rascanu_Tessitore_2003}
V.~Barbu, A.~R{\u{a}}{\c{s}}canu, and G.~Tessitore, ``Carleman estimates and
  controllability of linear stochastic heat equations,'' {\em Appl. Math.
  Optim.}, vol.~47, no.~2, pp.~97--120, 2003.

\bibitem{G16}
D.~Goreac, ``Approximate controllability for linear stochastic differential
  equations in infinite dimensions,'' {\em Applied Mathematics and
  Optimization}, vol.~60, no.~1, pp.~105--132, 2009.

\bibitem{Cohen_Elliott_SIAM_2011}
S.~N. Cohen and R.~J. Elliott, ``Backward stochastic difference equations and
  nearly time-consistent nonlinear expectations,'' {\em SIAM Journal on Control
  and Optimization}, vol.~49, no.~1, pp.~125--139, 2011.

\bibitem{yong_zhou_99}
J.~Yong and X.~Zhou, {\em Stochastic Controls. Hamiltonian Systems and HJB
  Equations}.
\newblock New York: Springer-Verlag, 1999.

\bibitem{CostaFragosoMarques}
O.~L.~V. Costa, M.~D. Fragoso, and R.~P. Marques, {\em Discrete-Time Markov
  Jump Linear Systems}.
\newblock Probability and Its Applications, Springer-Verlag London, 2005.

\bibitem{CostaOliveira_2012}
O.~L.~V. Costa and A.~de~Oliveira, ``Optimal mean--variance control for
  discrete-time linear systems with markovian jumps and multiplicative
  noises,'' {\em Automatica}, vol.~48, no.~2, pp.~304 -- 315, 2012.

\bibitem{NemhauserWolseyFisher_78}
G.~L. Nemhauser, L.~A. Wolsey, and M.~L. Fisher, ``An analysis of
  approximations for maximizing submodular set functions---i,'' {\em
  Mathematical Programming}, vol.~14, no.~1, pp.~265--294, 1978.

\bibitem{Lovasz:1983aa}
L.~Lov{\'a}sz, ``Submodular functions and convexity,'' in {\em Mathematical
  Programming The State of the Art} (A.~Bachem, B.~Korte, and M.~Gr{\"o}tschel,
  eds.), pp.~235--257, Springer Berlin Heidelberg, 1983.

\bibitem{SummersCortesiLygeros_2014}
T.~Summers, F.~Cortesi, and J.~Lygeros, ``On submodularity and controllability
  in complex dynamical networks,'' {\em IEEE Transactions on Control of Network
  Systems}, 2014.

\bibitem{hasty_pradines_dolnik_collins_00}
J.~Hasty, J.~Pradines, M.~Dolnik, and J.~Collins, ``Noise-based switches and
  amplifiers for gene expression,'' {\em PNAS}, vol.~97, no.~5, pp.~2075--2080,
  2000.

\bibitem{JiChizeck1988}
Y.~Ji and H.~J. Chizeck, ``Controllability, observability and discrete-time
  markovian jump linear quadratic control,'' {\em International Journal of
  Control}, vol.~48, no.~2, pp.~481--498, 1988.

\bibitem{GaySolimanFages2010}
S.~Gay, S.~Soliman, and F.~Fages, ``A graphical method for reducing and
  relating models in systems biology,'' {\em Bioinformatics}, vol.~26,
  pp.~i575--i581, 09 2010.

\bibitem{SolimanFagesRadulescu2014}
S.~Soliman, F.~Fages, and O.~Radulescu, ``A constraint solving approach to
  model reduction by tropical equilibration,'' {\em Algorithms for Molecular
  Biology}, vol.~9, no.~1, pp.~1--11, 2014.

\bibitem{MadelaineLhoussaineNiehren}
G.~Madelaine, C.~Lhoussaine, and J.~Niehren, ``Structural simplification of
  chemical reaction networks preserving deterministic semantics,'' in {\em
  Computational Methods in Systems Biology}, Lecture Notes in Computer Science,
  Springer, Sept. 2015.

\bibitem{C.M.GuldbergandP.Waage}
C.~Guldberg and P.~Waage, ``{S}tudies {C}oncerning {A}ffinity,'' {\em C. M.
  Forhandlinger: Videnskabs-Selskabet i Christiana}, vol.~35, 1864.

\bibitem{crudu_debussche_radulescu_09}
A.~Crudu, A.~Debussche, and O.~Radulescu, ``Hybrid stochastic simplifications
  for multiscale gene networks,'' {\em BMC Systems Biology}, p.~3:89, 2009.

\bibitem{cook_gerber_tapscott_98}
D.~L. Cook, A.~N. Gerber, and S.~J. Tapscott, ``Modelling stochastic gene
  expression: Implications for haploinsufficiency,'' {\em Proc. Natl. Acad.
  Sci. USA}, vol.~95, pp.~15641--15646, 1998.

\bibitem{Gillespie1977}
D.~T. Gillespie, ``Exact stochastic simulation of coupled chemical reactions,''
  {\em The Journal of Physical Chemistry}, vol.~81, no.~25, pp.~2340--2361,
  1977.

\bibitem{crudu_Debussche_Muller_Radulescu_2012}
A.~Crudu, A.~Debussche, A.~Muller, and O.~Radulescu, ``Convergence of
  stochastic gene networks to hybrid piecewise deterministic processes,'' {\em
  The Annals of Applied Probability}, vol.~22, pp.~1822--1859, 10 2012.

\bibitem{Costa_Dufour_08}
O.~L.~V. Costa and F.~Dufour, ``Stability and ergodicity of piecewise
  deterministic {M}arkov processes,'' {\em SIAM J. Control Optim.}, vol.~47,
  no.~2, pp.~1053--1077, 2008.

\bibitem{Goreac_SIAM_2015}
D.~Goreac, ``Asymptotic control for a class of piecewise deterministic markov
  processes associated to temperate viruses,'' {\em SIAM Journal on Control and
  Optimization}, vol.~53, no.~4, pp.~1860--1891, 2015.

\bibitem{CFJ_2014}
F.~Confortola, M.~Fuhrman, and J.~Jacod, ``Backward stochastic differential
  equations driven by a marked point process: an elementary approach, with an
  application to optimal control,'' {\em Annals of Applied Probability},
  vol.~to appear, 2015.
\newblock arXiv:1407.0876.

\bibitem{DaPratoZabczyk1992}
G.~D. Prato and J.~Zabczyk, {\em Stochastic Equations in Infinite Dimensions}.
\newblock Cambridge University Press, 1992.
\newblock Cambridge Books Online.

\end{thebibliography}

\end{document}